\documentclass{amsproc}

\usepackage{amsthm}
\usepackage{amssymb}
\usepackage{amsfonts}
\usepackage{amscd}
\usepackage[all]{xy}
\usepackage{graphics}

\theoremstyle{plain} 

\newtheorem{theorem}{Theorem}[section]

\newtheorem{corollary}[theorem]{Corollary}

\newtheorem{definition}[theorem]{Definition}
\newtheorem{proposition}[theorem]{Proposition}

\newtheorem{remark}[theorem]{Remark}

\usepackage{amsmath}
\usepackage{amsthm}
\usepackage{amscd}

\numberwithin{equation}{section}

\newcommand{\ba}{{\mathbf a}}
\newcommand{\bb}{{\mathbf b}}
\newcommand{\balpha}{{\boldsymbol \alpha}}

\newcommand{\bnu}{{\boldsymbol \nu}}

\newcommand{\bt}{{\mathbf t}}

\newcommand{\bl}{{\boldsymbol{l}}}
\newcommand{\bw}{{\boldsymbol{w}}}

\newcommand{\cc}{{\mathbb C}}
\newcommand{\pp}{{\mathbb P}}

\newcommand{\zz}{{\mathbb Z}}

\newcommand{\pc}{\pp^1_{\cc}}

\newcommand{\Pp}{{\mathcal P}}

\newcommand{\cO}{{\mathcal O}}
\newcommand{\Oo}{\mathcal O_{X}}

\newcommand{\Mm}{{\mathcal M}}
\newcommand{\Nn}{{\mathcal N}}

\newcommand{\cW}{{\mathcal W}}

\newcommand\rank{\mathop{\rm rank}\nolimits}

\newcommand\End{\mathop{\rm End}\nolimits}

\newcommand{\len}{\mathop{\rm length}\nolimits}
\newcommand{\res}{\mathop{\sf res}\nolimits}

\newcommand{\App}{\mathop{\rm App}\nolimits}
\newcommand{\Bun}{\mathop{\rm Bun}\nolimits}
\newcommand\lra{\longrightarrow}

\begin{document}

\author[F. Loray]{Frank Loray}
\address{CNRS, IRMAR, Campus de Beaulieu\\
Universit\'e de Rennes I\\
35042 Rennes Cedex, France}
\email{frank.loray@univ-rennes1.fr}

\author[M.-H. Saito]{Masa-Hiko Saito}
\address{Department of Mathematics, Graduate School of Science\\
Kobe University\\
Kobe, Rokko, 657-8501, Japan}
\email{mhsaito@math.kobe-u.ac.jp}
\thanks{This research was partly supported by JSPS Grant-in-Aid for Scientific Research (S)19104002, 24224001,  
challenging Exploratory Research  23654010}

\title[Lagrangian fibrations on moduli spaces]{Lagrangian fibrations in duality on moduli space
of rank two logarithmic connections over the projective line}

\subjclass[2010]{Primary 34M55; Secondary 14D20, 32G20, 32G34}

\keywords{Painlev\'e VI, Garnier systems, 
Logarithmic connection, Parabolic structure, Projective line, 
apparent singularity, Okamoto-Painlev\'e pair, Del Pezzo surface}

\begin{abstract}
We study the  moduli space of logarithmic connections of rank $2$
on $\pp_{\cc}^1$ minus $n$ points with fixed spectral data. 
There are two natural Lagrangian maps: one towards apparent singularities
of the associated fuchsian scalar equation, and another one towards
moduli of parabolic bundles. We show that these are transversal and dual to each other.
In case $n=5$, we recover the beautiful geometry of Del Pezzo surfaces of degree $4$.
\end{abstract}

\maketitle

\section{Introduction}

In this paper, we investigate the geometry of moduli space of rank $2$
logarithmic connections over the Riemann sphere and extend some results 
obtained together with Carlos Simpson in the $4$-point case \cite{LoraySaitoSimpson}.
Precisely, we fix a reduced effective divisor $D=t_1+\cdots+t_n$ on $X:=\mathbb P^1_{\mathbb C}$
and consider those pairs $(E,\nabla)$ where $E$ is a rank $2$ vector bundle
over $X$ and $\nabla:E\to E\otimes\Omega^1_{X}(D)$
a connection having simple poles supported by $D$. At each pole, we have two residual eigenvalues 
$\{\nu_i^+,\nu_i^-\}$, $i=1,\ldots,n$; they satisfy Fuchs relation $\sum_{i}(\nu_i^++\nu_i^-)+d=0$ where $d=\deg(E)$. Moreover, we can naturally introduce 
parabolic strucures $\bl= \{ l_i \}_{1 \leq i \leq n}$ such that $l_i$ is a one dimensional subspace of $E_{|t_i}$ which corresponds to an eigen space of  the residue of $\nabla$ at $t_i$ with the eigenvalue $\nu_{i}^+$.  Note that when $\nu_i^+\not=\nu_i^-$,
the parabolic structure $\bl$ is determined by the connection $(E,\nabla)$.
Fixing  spectral data $\boldsymbol{\nu}=(\nu_i^\pm)$ with integral sum $-d$, 
by introducing the weight $\bw$ for stablity, we can construct 
the moduli space $M^{\bw}(\bt,\boldsymbol{\nu})$ of $\bw$-stable 
$\bnu$-parabolic connections $(E,\nabla, \bl)$ by Geometric Invariant Theory \cite{InabaIwasakiSaito} and the moduli space $M^{\bw}(\bt,\boldsymbol{\nu})$ turns to be a smooth irreducible quasi-projective variety of dimension $2(n-3)$.
It is moreover equipped with a natural holomorphic symplectic structure. 
We note that, when $\sum_{i=1,\ldots,n}\nu_i^{\epsilon_i}\not\in \zz$, 
for any choice $(\epsilon_i) \in \{+,-\}^n $, every parabolic connection $(E,\nabla, \bl)$ is irreducible,
and thus stable; the moduli space $M^{\bw}(\bt,\boldsymbol{\nu})$ does not depend on the choice of weights $\bw$
in this case. 
These moduli spaces occur as space of initial conditions for Garnier systems, the case $n=4$ corresponding
to the Painlev\'e VI equation (cf. \cite{JimboMiwa,InabaIwasakiSaito2}).

There are many isomorphisms between these moduli spaces. For instance, 
twisting by a rank $1$ connection (with the same poles), one can translate
the spectral data as $(\nu_i^{\pm})\mapsto(\nu_i^{\pm}+\mu_i)$ with the only
restriction $\sum_i \mu_i\in\mathbb Z$. Also, 
by using elementary (or Schlesinger) transformations, we may shift each
$\nu_i^{\pm}$ by arbitrary integer, thus freely shifting the degree $d$ of vector bundles. 
In particular, it is enough for our purpose
to consider the case where $\sum_{i}(\nu_i^++\nu_i^-) =1$, which means by Fuchs relations
that $d=\det(E)=-1$.

\subsection{Apparent map}
There are two natural Lagrangian fibrations on these moduli spaces. 
One of them is given by the ``apparent map'' 
$$
\mathrm{App}:M^{\bw} (\bt,\bnu)\dashrightarrow \vert \mathcal O_{X}(n-3)\vert \simeq {\mathbb P}^{n-3}_{\mathbb C},
$$ 
which is a rational dominant map towards the projective space of the linear system
(see \cite{DubrovinMazzocco,InabaIwasakiSaito2, SaitoSzabo}). 
Here, we need to fix degree $d=-1$.
The image $\mathrm{App}(E,\nabla)$
is defined by the zero divisor of the composite map
$$\mathcal O_{X}\to E\stackrel{\nabla}{\longrightarrow} E\otimes\Omega^1_{X}(D)\to (E/\mathcal O_{X})\otimes\Omega^1_{X}(D).$$
For a generic connection $(E,\nabla)$, it is well-known that the underlying bundle is $E=\mathcal O_{X}\oplus\mathcal O_{X}(-1)$ and the map $\mathcal O_{X}\to E$ is therefore unique up to 
scalar multiplication; the right-hand-side arrow is just the quotient by the image of $\mathcal O_{X}\to E$.
The apparent map is therefore well-defined on a large open set of $M^{\bw} (\bt,\boldsymbol{\nu})$.
Choosing the image of $\mathcal O_{X}$ as a cyclic vector allow to derive a $2nd$ order
Fuchsian differential equation; $\mathrm{App}(E,\nabla)$ gives the position of extra apparent singular points arising
from this construction, whence the name. The apparent map $\mathrm{App}$ has indeterminacy points where $E\not\simeq\mathcal O_{X}\oplus\mathcal O_{X}(-1)$ or
$E \simeq \mathcal O_{X}\oplus\mathcal O_{X}(-1)$ and $\mathcal O_{X}$ is invariant under $\nabla$.  
 
\subsection{Parabolic fibration}

Let $\Pp_d(\bt)$ be the moduli stack of undecomposable  parabolic bundles of degree $d$,  and denote by  $P_d(\bt)$ its corresponding coarse moduli space.  It is known that the natural map 
$\Pp_d(\bt) \lra P_d(\bt)$ is a ${\bf G}_m$-gerbe.  
We assume that $\bnu$ is generic  such as  undecomposable parabolic bundles of degree $d$ 
coincide with the $\bnu$-flat parabolic bundles, i.e, those admitting a connection with prescribed parabolic and spectral data $(\bl,\boldsymbol\nu)$ (see Proposition \ref{prop:flat=undecomposable}).  
Under this assumption,  we can define the second fibration, which is  more natural, but actually more subtle to define. It is the forgetfull map 
$(E,\nabla, \bl) \mapsto (E,\bl)$
towards the coarse moduli space $P_d(\bt)$  (here we do not need anymore $d=-1$). Always assuming generic spectral data $\boldsymbol\nu$,
the parabolic structure $\bl=\{l_i\}$ is the data of the residual eigendirection $l_i\subset E_{|t_i}$ of $\nabla$ with respect to the eigenvalue $\nu_i^+$ for each pole $t_i$. However, as observed in \cite{ArinkinLysenko1,ArinkinLysenko2},
the coarse moduli space $P_d(\bt)$  is a non Hausdorff topological space, or a nonseparated scheme. 
To get a nice moduli space, we have to 
impose a stability condition with respect to weights $\bw=(w_i)\in[0,1]^n$ (see \cite{MehtaSeshadri});
the moduli space $P_d^{\bw}(\bt)$ of $\bw$-semistable
degree $d$ parabolic bundles $(E,\bl)$ is therefore a normal irreducible projective variety; the open subset of 
$\bw$-stable bundles is smooth. These moduli spaces actually depend on the choice of 
weights. For generic weights, $\bw$-semistable=$\bw$-stable and we get a 
smooth projective variety. Precisely, there are finitely many hyperplanes (walls) cutting out $[0,1]^n$ into
finitely many chambers and strictly $\bw$-semistable bundles only occur along walls.
The moduli space $P_d^{\bw}(\bt)$ is locally constant in each chamber and is either 
empty, or has expected dimension $n-3$. It becomes singular along walls (or maybe reduced to a single point).
A generic $\bw$-stable connection is also
$\bw'$-stable for any weight $\bw'$ (for which the corresponding moduli space has the right dimension).
We thus get natural birational maps $P_d^{\bw}(\bt)\stackrel{\sim}{\dashrightarrow}P_d^{\bw'}(\bt)$ identifying generic bundles that occur in both moduli spaces.
An important fact (see Proposition \ref{prop:undecomposable=stable})
is that a parabolic bundle $(E,\bl)$
is undecomposable   if, and only if, it is $\bw$-stable for a convenient choice of weights. As  a consequence,
the coarse moduli space $P_{d}(\bt)$ 
of undecomposable  parabolic bundles of degree $d$  
is covered by those smooth projective varieties 
$P^{\bw}_d(\bt)$ when $\bw$ runs over all chambers (cf. Proposition \ref{prop:modulistack}): 
\begin{equation}\label{eq:patchintro}
P_{d}(\bt) = \cup_{i, finite } P_d^{\bw_{i}}(\bt).
\end{equation} 
In fact, the coarse moduli space is obtained by patching together these (non empty) projective charts
along (strict) open subsets; each of these projective charts are open and dense in the coarse moduli space.
Once we choose one of these charts, we get a rational map 
$M^{\bw} (\bt,\boldsymbol{\nu})\dashrightarrow P_d^{\bw}(\bt)$
which turns to be Lagrangian, like for $\mathrm{App}$ in case $d=-1$. Moreover  we can extend this rational map 
to a rational map  
$\mathrm{Bun}:M^{\bw} (\bt,\boldsymbol{\nu}) \dashrightarrow P_{d}(\bt)$, which turns to be 
a morphism when $\bnu$ is generic.   

\subsection{Results}

Assuming from now on $d=-1$, from the two rational maps $\mathrm{App}$ and $\mathrm{Bun}$, we obtain 
the rational map
\begin{equation}\label{eq:rational-map}
\mathrm{App}\times\mathrm{Bun}:M^{\bw}(\bt,\boldsymbol{\nu})\dashrightarrow
\vert\mathcal O_{X}(n-3)\vert\times P^{\bw}_{-1}(\bt).   
\end{equation}
In this paper, we will basically prove that this map is birational provided that $\sum_i \nu_i^-\not=0$.
However we will be able to give more precise information about 
the rational map (\ref{eq:rational-map}) by introducing a  choice of democratic weights  
$\bw$ (see (\ref{eq:democratic})) and a 
good open subset $M^{\bw}(\bt,\boldsymbol{\nu})^{0}\subset M^{\bw}(\bt,\boldsymbol{\nu})$.

For such a choice of weights $\bw$ in (\ref{eq:democratic}), 
$\bw$-stable parabolic bundles $(E,\bl)$ are precisely those flat
bundles for which $E=\mathcal O_{X}\oplus\mathcal O_{X}(-1)$ and none of the parabolics coincide 
with the special subbundle: $l_i\not\in\mathcal O_{X}$ for all $i=1,\ldots,n$. 
We are moreover able to construct a natural isomorphism $P_{-1}^{\bw}(\bt)\stackrel{\sim}{\to}\vert\mathcal O_{X}(n-3)\vert^{*}$
with the dual of the linear system involved in the apparent map (cf. Proposition \ref{Prop:OurMainChart}).
We therefore introduce the open subset  
$$M^{\bw} (\bt,\boldsymbol{\nu})^{0} :=\mathrm{Bun}^{-1}P_{-1}^{\bw}(\bt)\ \subset\ M^{\bw} (\bt,\boldsymbol{\nu}) $$
by imposing the conditions on $(E, \nabla, \bl)$  that $(E, \bl) \in P_{-1}^{\bw}(\bt)$. 
Then the two rational maps $\mathrm{App}$ and $\mathrm{Bun}$ now induce a natural morphism
\begin{equation}\label{eq:introd-map}
\mathrm{App}\times\mathrm{Bun}:M^{\bw} (\bt,\boldsymbol{\nu})^{0} \lra 
\vert\mathcal O_{X}(n-3)\vert\times \vert\mathcal O_{X}(n-3)\vert^{*} 
\end{equation}
and both $\mathrm{App}$ and $\mathrm{Bun}$ are Lagrangian.
We can state our result as follows.  

\begin{theorem}\label{MainTheoremIntro} Under the assumption that  $\sum_i \nu_i^-\not=0$, 
 the morphism $\mathrm{App}\times\mathrm{Bun}$ in (\ref{eq:introd-map})  
is an open embedding and its image  coincides with 
the complement of the incidence variety $\Sigma 
\subset \vert\mathcal O_{X}(n-3)\vert\times \vert\mathcal O_{X}(n-3)\vert^{*}$ for the duality.  
\end{theorem}

In order to make the statement of Theorem \ref{MainTheoremIntro} more precise, let us 
introduce the homogeneous coordinates 
$\ba=(a_0:\cdots:a_{n-3})$ on
$\vert\mathcal O_{X}(n-3)\vert \simeq \pp_{\ba}^{n-3}$, and the dual coordinates 
$\bb=(b_0:\cdots:b_{n-3})$ on $\vert\mathcal O_{X}(n-3)\vert^{*}\simeq  \pp_{\bb}^{n-3}$. 
Let $\Sigma \subset \pp_{\ba}^{n-3} \times \pp_{\bb}^{n-3}$ be  
the incidence variety, whose  defining equation is given by $\sum_ka_kb_k=0$.
Then the morphism $\App \times \Bun$ induces the isomorphism (see Theorem \ref{MainTheoremIntro})
$$
\App \times \Bun: M^{\bw}(\bt,\bnu)^0 \stackrel{\sim}{\lra}
 (\pp_{\ba}^{n-3} \times \pp_{\bb}^{n-3}) \setminus \Sigma.
$$
Setting $\rho:=-\sum_i\nu_i^-$, the symplectic structure of $M(\bt,\boldsymbol{\nu})^0$ is given by 
$$\omega=d\eta\ \ \ \text{where}\ \ \ \eta
=\rho\frac{a_0db_0+\cdots+a_{n-3}db_{n-3}}{a_0b_0+\cdots+a_{n-3}b_{n-3}}.$$ 
When $\rho=0$, we prove that $\mathrm{App}\times\mathrm{Bun}$
degenerates: it is dominant onto $\Sigma$.

In order to prove Theorem \ref{MainTheoremIntro}, we intoduce a 
good compactification $\overline{M^{\bw}(\bt, \bnu)^{0}} \supset M^{\bw} (\bt,\boldsymbol{\nu})^{0}$
(cf. Section \ref{sec:compactmoduli}) which turns to be another moduli space:
\begin{equation}\label{eq:intro-compactification}
\overline{M^{\bw}(\bt, \bnu)^{0}} := \left\{ 
\begin{array}{l}(E, \nabla, \lambda \in \cc, \bl) \\ 
\mbox{$\lambda$-$\bnu$-parabolic connection} 
\end{array}
\ | \ \mbox{ $ (E, \bl) \in  
 P^{\bw}_{-1}(\bt)$ }  \right\}/ \simeq 
\end{equation}
where equivalence $\simeq$ is given by bundle isomorphisms and the natural $\mathbb C^*$-action by scalar multiplication.
The open subset $M^{\bw}(\bt, \bnu)^{0} \hookrightarrow \overline{M^{\bw}(\bt, \bnu)^{0}}$ is given by those $\lambda$-$\bnu$-parabolic connections for which $\lambda\not=0$, and the complement 
$$
M^{\bw}(\bt, \bnu)^0_H := 
\overline{M^{\bw}(\bt, \bnu)^{0}} \setminus M^{\bw}(\bt, \bnu)^{0} 
$$ 
is the moduli space of $\bw$-stable parabolic Higgs bundles.  
Now Theorem \ref{MainTheoremIntro} easily follows from the following (cf. Theorem \ref{th:Main})

\begin{theorem}\label{thm:MainMainThmIntro} 
If $\sum_i \nu_i^-\not=0$, the moduli space $\overline{M^{\bw}(\bt, \bnu)^0}$ is a smooth projective variety and we can extend the morphism  (\ref{eq:introd-map}) as 
an isomorphism
$$
\mathrm{App}\times\mathrm{Bun}:\overline{M^{\bw}(\bt,\bnu)^0}\stackrel{\sim}{\lra}
\vert \mathcal O_{X}(n-3) \vert \times  \vert \mathcal O_{X}(n-3) \vert^{\ast}
$$
Moreover, by restriction, we also obtain the isomorphism 
$$
\mathrm{App}\times\mathrm{Bun}_{|M^{\bw}(\bt, \bnu)_H^0} : M^{\bw}(\bt, \bnu)_H^0\stackrel{\sim}{\lra} \Sigma
$$
where $\Sigma \subset  \vert \mathcal O_{X}(n-3) \vert \times  \vert \mathcal O_{X}(n-3) \vert^{\ast}$ 
is the incidence variety for the duality.  
\end{theorem}

Here we note that the coarse moduli space of $\bw$-stable $\lambda$-$\bnu$ parabolic connections 
without the condition in (\ref{eq:intro-compactification}) have singularities.  So 
our choice of the weight $\bw$ and the compactification in (\ref{eq:intro-compactification}) 
is essential to prove Theorem \ref{thm:MainMainThmIntro}.

It is well-known \cite{ArinkinLysenko1,ArinkinLysenko2} that the moduli space $M^{\bw}(\bt,\boldsymbol{\nu})$
should be an affine extension of the cotangent bundle of $P_{-1}(\bt)$; once we have 
choosen a projective chart $P_{-1}^{\bw}(\bt)\simeq\mathbb P_{\mathbf{b}}^{n-3}$, the restricted
affine bundle $M^{\bw}(\bt,\boldsymbol{\nu})^0$ must be either the cotangent bundle $T^*\mathbb P_{\mathbf{b}}^{n-3}$,
or the unique non trivial affine extension of $T^*\mathbb P_{\mathbf{b}}^{n-3}$. 
Here, we prove that we are in the latter case if, and only if, $\rho\not=0$.
But the nice fact is that apparent map provides a natural trivialization for the compactification
of this affine bundle.

The projective space $\vert\mathcal O_{X}(n-3)\vert\simeq\mathbb P_{\mathbf{a}}^{n-3}$
may be considered as the space of polynomial equations $\sum_ka_kz^k=0$. We can consider
the $(n-3)!$-fold cover $(\pc)^{n-3}\to (\pc)^{(n-3)}:= (\pc)^{n-3}/{\mathfrak S}_{n-3}=\mathbb P_{\mathbf{a}}^{n-3}$ 
parametrized by ordered roots $(q_1:\cdots:q_{n-3})$. 
Since we have a morphism $\mathrm{App}:M(\bt, \bnu)\dashrightarrow \mathbb P_{\mathbf{a}}^{n-3}$, 
by the fibered product,
we get a $(n-3)!$-fold cover $\widetilde{M^{\bw}(\bt,\bnu)^0}\lra M^{\bw} (\bt,\bnu)^0$. 
This latter one (or some natural partial compactification $\widetilde{M^{\bw} (\bt,\bnu)}$) has been described
in many papers \cite{Okamoto,DubrovinMazzocco,Suzuki, SaitoSzabo}. The space $\widetilde{M(\bt,\bnu)}$
can be covered by affine charts $\cc^{2n-6}$ with Darboux coordinates $(p_k,q_k)$  
for the symplectic structure: $\omega=\sum_k dp_k\wedge dq_k$.
Our parameters can be expressed in terms of symmetric functions of $p_k$'s and $q_k$'s,
what we do explicitely in the five pole case $n=5$ at the end of the paper.
From this point of view, S. Oblezin constructs in \cite{Oblezin} a natural birational map 
$M(\bt,\bnu)\dashrightarrow (K_n)^{(n-3)}$ where $K_n$ is an open subset
of the total space of $\Omega^1_X(D)$ blown up at $2n$ points.
Precisely, at each fiber $z=t_i$, we have the residual map
$$\Omega^1_X(D)_{\vert t_i}\stackrel{\mathrm{Res}_{t_i}}{\lra}\cc$$
and we blow-up the two points corresponding to $\nu_i^+,\nu_i^-\in\cc$ in the fiber;
then we delete the strict transforms of fibers $z=t_i$ to obtain the open set $K_n$. 
So far, no natural system of coordinates was provided on $M^{\bw}(\bt,\bnu)$ for $n\ge5$.

In \cite{Arinkin}, 
Arinkin investigated the geometric Langlands problem related to the case of $n=4$ 
by using the natural morphism of coarse modui spaces $\Bun:M(\bt, \bnu) \lra P_{-1}(\bt)$ (see also \cite{ArinkinFedorov}).
Though we cannot directly extend his methods to the case of $n > 4$, we  may expect that our main Theorem \ref{MainTheoremIntro} and \ref{thm:MainMainThmIntro} may give some approach to obtain a similar result.

In the last part of the paper, we investigate with many details the case $n=5$. 
We provide a precise description of the non-separated coarse moduli space
$P_{-1}(\bt)$ of undecomposable parabolic bundles, which turns to be closely related 
to the geometry of degree $4$ Del Pezzo surfaces. Precisely,
there is a natural embedding $X\hookrightarrow V:=\mathbb P^2_\bb$ as a conic.
We then consider the blow-up $\phi:\hat V\to V$
of the images of the $5$ points $t_1,\ldots,t_5\in X$: this is the Del Pezzo surface
associated to our problem. If we
denote by $\Pi_i\subset \hat V$ the exceptional divisor over $t_i$, and by $\Pi_{i,j}\in\hat V$
the strict transform of the line in $V$ passing through $t_i$ and $t_j$ for any $i,j$,
then these $\Pi,\Pi_i,\Pi_{i,j}$ are 
the well-known $16$ rational curves with self-intersection $(-1)$ in $\hat V$. 
For any $i=1,\ldots,5$, by contracting all five $(-1)$-curves intersecting $\Pi_i$,
we get a new morphism $\phi_i:\hat V\to V_i\simeq\mathbb P^2_\cc$.

\begin{theorem}\label{prop:patching5}The coarse moduli space $P_{-1}(\bt)$ is given by:
$$
P_{-1}(\bt)=\hat V \cup V\cup V_1\cup V_2\cup V_3\cup V_4\cup V_5
$$
where $\cup$ means that we are patching these projective manifolds 
by means of the birational maps $\phi$, $\phi_i$ and $\phi_i\circ\phi_j^{-1}$
along the maximal open subsets where they are one-to-one.
\end{theorem}

For instance, $\phi:\hat V\to V$ induces an isomorphism
$$\hat V\setminus (\Pi_1\cup \Pi_2\cup \Pi_3\cup \Pi_4\cup \Pi_5)\stackrel{\sim}{\lra} V\setminus\{t_1,t_2,t_3,t_4,t_5\}$$
and we patch $V$ to $\hat V$ along these open subsets by means of this isomorphism. 
Moreover, all these projective charts $\hat V,V,V_i$ are realized as coarse moduli spaces
of stable parabolic bundles $P_{-1}^{\bw}(\bt)$ for convenient choices of weights $\bw$
and, in the patching, we just identify all isomorphism classes of bundles that are shared 
by any two of these projective charts.
Finally, we explain how to recover the total moduli space $M(\bt,\boldsymbol{\nu})$
by blowing-up $\vert\mathcal O_{X}(n-3)\vert\times \vert\mathcal O_{X}(n-3)\vert^{*}$ 
along some lifting of these curves inside the incidence variety $\Sigma$.

{\bf Acknowledgement:} The authors would like to warmly thank Michele Bolognesi for useful discussions about moduli spaces of $n$-points configurations on the projective line.

\section{Moduli space of connections}\label{Sect:ModuliConnection}

In this section, 
we will recall some results in \cite{Nitsure,MehtaSeshadri,Seshadri,MaruyamaYokogawa,InabaIwasakiSaito,Inaba}. 

\subsection{Definition of the moduli space (as geometric quotient)}

Let us fix a set of $n$-distinct points $\bt=\{ t_1, \cdots, t_n \}$ on the Riemann sphere $X:=\pp^1_{\cc}$
and define the divisor $D= t_1+ \cdots+ t_n$. In this paper, a logarithmic connection of rank $2$ on $X$ 
 with singularities (or poles) at $D$ is a pair $(E, \nabla)$ consisting of  an algebraic (or holomorphic) 
vector bundle $E$ on $\pc$ of rank $2$ and  a linear algebraic connection  
$\nabla:E \lra E \otimes \Omega^1_{X}(D)$.  
We can define the residue homomorphism 
$\res_{t_i}(\nabla) \in \End(E_{|t_i}) \simeq M_2(\cc)$ and then 
let $\nu_i^+, \nu_i^-$ be the eigenvalues of $\res_{t_i}(\nabla)$, that we call {\em local exponents}. 
Fuchs relation says that $\sum_i (\nu_i^++\nu_i^-) = - \deg E =-d$. 
So we define the set of local exponents of degree $d$
\begin{equation}\label{eq:exponent}
\Nn_{n}(d) := \left\{ \bnu = (\nu_i^{\pm})_{1\leq i \leq n} \in \cc^{2n} \quad  \left| 
 \quad d + \sum_{1\leq i \leq n}(\nu_i^++\nu_i^-) = 0 \right\} \simeq 
 \cc^{2n-1} \right.
\end{equation}

\begin{definition}\label{def:ParabolicConnection}
 Fix $\bnu \in \Nn_{n}(d)$.  A $\bnu$-parabolic connection 
 on $(X,D)=(\pc, \bt) $ is a collection $(E, \nabla, \bl=\{ l_i \} )$ consisting of the following data:
 \begin{enumerate}
 \item a logarithmic connection $(E, \nabla)$ on $(X,D)$ of rank $2$ with spectral data $\bnu$,  
 \item a one dimensional subspace $l_i\subset E_{|t_i}$ 
on which $\res_{t_i}(\nabla)$ acts by multiplication by $\nu_i^+$.
 \end{enumerate}
\end{definition}
For generic $\bnu$, the parabolic direction $l_i$ is nothing but the eigenspace for $\res_{t_i}(\nabla)$ 
with respect to $\nu_i^+$
so that the parabolic data is uniquely defined by the connection $(E,\nabla)$ itself. However,
when $\nu_i^+=\nu_i^-$ and $\res_{t_i}(\nabla)$ is scalar (i.e. diagonal), then 
the parabolic $l_i$ add a non trivial data and this allow to avoid singularities 
of the moduli space.

In order to obtain a good moduli space, 
we have to introduce a stability condition for parabolic connections.  
For this, fix weights $\bw=(w_1,\ldots,w_n)\in[0,1]^n$. Then 
for any line subbundle $F\subset E$, define the $\bw$-stability index 
of $F$ to be the real number
\begin{equation}\label{def:StabIndex}
\mathrm{Stab}(F):= \deg(E)- 2 \deg (F) 
+ \sum_{l_i \not= F_{|t_i}}w_i-\sum_{l_i=F_{|t_i}} w_i.
\end{equation}

\begin{definition}\label{def:stable}
A $\bnu$-parabolic connection $(E,\nabla, \bl )$ will be called {\em $\bw$-stable} 
(resp.  {\em $\bw$-semistable}) if
for any rank one $\nabla$-invariant subbundle $F \subset E$,  
$$\nabla(F)\subset F\otimes\Omega_{X}^1(D)$$
the following inequality holds
\begin{equation}\label{eq:stability}
 \mathrm{Stab}(F)>0 \ \ \ (\text{resp.}\ \ge0).
\end{equation}
\end{definition}

A rank $2$ parabolic bundle $(E,\bl)$ is called {\em $\bw$-stable} (resp.  {\em $\bw$-semistable})
if inequality (\ref{eq:stability}) holds for any rank one subbundle $F\subset E$. In particular,
a $\bnu$-parabolic connection $(E,\nabla, \bl )$ may be stable while the underlying
parabolic bundle $(E,\nabla)$ is not. For instance, if  $(E,\nabla, \bl )$ is irreducible,
there is no strict $\nabla$-invariant subbundle and condition
(\ref{eq:stability}) is just empty.

\begin{remark}{\rm To fit into usual notations (see \cite{MehtaSeshadri,Seshadri,MaruyamaYokogawa,Inaba,InabaIwasakiSaito}), one should rather consider the flag
$ \{l^{(i)}_0  \supset l^{(i)}_1  \supset l^{(i)}_2\}:=\{E_{|t_i}\supset l_i\supset \{0\}\}$ and ask 
that $(\res_{t_i}(\nabla)-\nu^{(i)}_j Id)(l^{(i)}_j) = l^{(i)}_{j+1}$ for each $j=0, 1$,
where $(\nu^{(i)}_0,\nu^{(i)}_1):=(\nu_i^-,\nu_i^+)$. In the rank $2$ case, this is 
equivalent to Definition \ref{def:ParabolicConnection}.
Then, weights rather look like:
$$
\balpha=(\alpha^{(1)}_1,\alpha^{(1)}_2,\ldots,\alpha^{(n)}_1,\alpha^{(n)}_2) 
\in \mathbb R^{2n}
$$
satisfying $ \alpha^{(i)}_1 \leq \alpha^{(i)}_2 \leq \alpha^{(i)}_1+1 $. 
Then, for any nonzero $\nabla$-invariant subbundle $F\subset E$, we define integers 
$$\len (F)^{(i)}_j = \dim (F|_{t_i}\cap l^{(i)}_{j-1})/(F|_{t_i}\cap l^{(i)}_j),$$
and stability is defined by the inequality
$$ \deg F + \sum_{i=1}^n \sum_{j=1}^2
 \alpha^{(i)}_j \len(F)^{(i)}_j
 < \frac{\deg E + \sum_{i=1}^n \sum_{j=1}^2 \alpha^{(i)}_j 
 \len(E)^{(i)}_j}{2}.
$$
However, it is straightforward to check that this condition is equivalent to (\ref{eq:stability})
after setting $w_i=\alpha^{(i)}_2-\alpha^{(i)}_1$. Sometimes, a parabolic degree is defined by
$$
\deg^{\mathrm{par}}F:= \deg F + \sum_{i=1}^n \sum_{j=1}^2
 \alpha^{(i)}_j \len(F)^{(i)}_j =  
\deg(F)+ \sum_{i=1}^n \alpha^{(i)}_{\epsilon_i}
$$
(including the case $F=E$) and in this case, the stability index for a line bundle is given by
$\mathrm{Stab}(F):= \deg^{\mathrm{par}}E-2\deg^{\mathrm{par}}F$. We also note
that, in \cite{InabaIwasakiSaito}, it was assumed that $0 < \alpha^{(i)}_2 < \alpha^{(i)}_1 < 1$ 
and genericity conditions to obtain the smooth moduli space for all $\bnu \in \Nn_n(d)$ 
simultaneously.  In this paper, we will vary the weights $\balpha$ and may consider the case of 
the equality $\alpha^{(i)}_1 = \alpha^{(i)}_2 $ for some $i$. Note also that in \cite{LoraySaitoSimpson}, 
we use the minus sign in front of the weight $\alpha^{(i)}_j$.}
\end{remark}

Consider the line bundle $L:=\mathcal O_{X}(d)$ and denote by 
$\nabla_{L}:L \lra L \otimes \Omega^1_{X}(D)$
the unique logarithmic connection having residual eigenvalue $\nu_i^++\nu_i^-$
at each pole $t_i$. For any $\bnu$-parabolic connection $(E,\nabla,  \bl)$ like above, 
there exists an isomorphism $\varphi: \wedge^2 E \lra L$; it is unique up to a scalar and 
automatically conjugates the trace
connection $\mathrm{tr}(\nabla)$ with $\nabla_{L}$.
We must add a choice of such isomorphism in the data $(E,\nabla, \varphi, \bl)$
in order to kill-out automorphisms; this is needed in the construction of the moduli space. 
We omit this data $\varphi$ in the sequel for simplicity.

Define the moduli space 
$M^{\bw}(\bt,\bnu)$
of isomorphism classes of $\bw$-stable $\bnu$-parabolic connections 
$(E,\nabla,\bl)$.  
Set $T_n = \{ \bt=(t_1, \cdots, t_n) \in X^n\ ;\ t_i \not= t_j\ \text{for}\ i \not= j \}$. 
Considering the relative setting of moduli space over $T_n \times \Nn_n(d)$, we obtain a family of moduli space
$$
\pi_n: M^{\bw} \lra T_n \times \Nn_n(d), 
$$ 
such that
$
M^{\bw}(\bt,\bnu) = \pi_n^{-1}(\bt, \bnu). 
$

\begin{theorem}\label{thm:fund} {\rm ([Theorem 2.1, Proposition 6.1,  \cite{InabaIwasakiSaito}]).} Assume that $n > 3$.  
For a generic  weight $\bw$, we can construct a 
relative fine moduli space 
$$\pi_n:M^{\bw}  \lra T_n \times \Nn_n(d),$$
which is a {\em smooth, quasi-projective} morphism of relative dimension $2n-6$ with irreducible closed fibers.  Therefore, the moduli space 
$M^{\bw}(\bt, \bnu)$ is a {\em smooth, irreducible} quasi-projective algebraic 
variety of dimension $2n -6$  for all 
$(\bt, \bnu) \in T_n \times \Nn_n(d)$.  Moreover the moduli space $M^{\bw}(\bt, \bnu)$ admits a natural holomorphic symplectic structure.  
\end{theorem}

\subsection{Isomorphisms between moduli spaces: twist and elementary transformations}\label{sec:twist&elm}
Given a connection $(F,\nabla_F)$ of rank 1, 
$$\nabla_F :F\to F\otimes \Omega_X^1(D)$$
with local exponents $(\mu_1,\ldots,\mu_n)$, we can define the {\em twisting map}
$$
\otimes(F,\nabla_F): \left\{\begin{matrix}
 M^{\bw}(\bt,\bnu)&\to&M^{\bw}(\bt,\bnu')  \\
  (E, \nabla,\bl)&\mapsto& (E\otimes F, \nabla\otimes\nabla_F, \bl)
\end{matrix}\right.
$$
where $\boldsymbol{\nu}'=(\nu_i^{\pm}+\mu_i)$. It is an isomorphism. It follows that
our moduli space only depend on differences $\nu_i^+-\nu_i^-$. 
On the other hand, this allow to rather freely modify
the trace connection. Precisely, depending on the parity of the degree $d$,
we can go into one of the following two cases
\begin{itemize}
\item in the even case, $(L,\nabla_L)=(\cO_X,d)$ and $(\nu_i^+,\nu_i^-)=(\frac{\kappa_i}{2},-\frac{\kappa_i}{2})$;
\item in the odd case, $(L,\nabla_L)=(\cO_X(-t_n),d+\frac{dz}{z-t_n})$ and 

$(\nu_i^+,\nu_i^-)=(\frac{\kappa_i}{2},-\frac{\kappa_i}{2})$
except $(\nu_n^+,\nu_n^-)=(\frac{\kappa_n}{2}+\frac{1}{2},-\frac{\kappa_n}{2}+\frac{1}{2})$;
\end{itemize}
where $(L,\nabla_L)$ is the fixed trace connection as above.



For each $i=1,\ldots,n$, we can define the {\em elementary transformation}
$$\mathrm{Elm}_{t_i}^-:\left\{\begin{matrix}
M^{\bw} (\bt,\bnu)&\to& M^{\bw'}(\bt, \bnu')\\
(E, \nabla,\bl)&\mapsto&(E', \nabla', \bl')
\end{matrix}\right.
$$
The vector bundle $E'$ is defined by the exact sequence of sheaves
$$0\lra  E'\lra  E\lra  E/l_i\lra 0$$ 
where $l_i$ is viewed here as a sky-scrapper sheaf.
The parabolic direction $l_i'$ is therefore defined as the kernel of the natural morphism
$E'\lra E$. The new connection $\nabla'$ is deduced from the action of $\nabla$
on the subsheaf $ E'\subset E$ and, over $t_i$, eigenvalues are changed by
$$(\nu_i^+,\nu_i^-)'=(\nu_i^-+1,\nu_i^+)\ \ \ (\text{and other}\ \nu_j^{\pm}\ \text{are left unchanged for}\ j \not= i).$$
If a line bundle $F\subset E$ contains the parabolic $l_i$, it is left unchanged
and we get $F\subset E'$; on the other hand, when $l_i\not=F_{|t_i}$ then 
we get $F'\subset E'$ with $F'=F\otimes\cO_X(-t_i)$. It follows that stability
condition is preserved for $\bw'$ defined by 
$$w_i'=1-w_i\ \ \ (\text{and other}\ w_j\ \text{are left unchanged for}\ j \not= i).$$
The composition $\mathrm{Elm}_{t_i}^-\circ\mathrm{Elm}_{t_i}^-$ is just the twisting map 
by $(F,\nabla_F)=(\mathcal O_{X}(-t_i),d+\frac{dz}{z-t_n})$.
We may also define $\mathrm{Elm}_{t_i}^+$ as the inverse of $\mathrm{Elm}_{t_i}^-$:
$$\mathrm{Elm}_{t_i}^+=\mathrm{Elm}_{t_i}^-\otimes\left(\mathcal O_{X}(t_i),d-\frac{dz}{z-t_i}\right).$$


Although we are mainly interested in the degree $d=-1$ case where the two Lagrangian fibrations
naturally occur, we will also consider the degree $0$ case of $\mathfrak{sl}_2$-connections
to compare with \cite{ArinkinLysenko1,ArinkinLysenko2,Oblezin}. Explicit computations 
will be made by means of 
$$\mathrm{Elm}_{t_n}^-:M(\bt,\bnu)\to M(\bt,\bnu')$$
going from degree $0$ to degree $-1$ case.

\section{The coarse moduli of undecomposable quasi-parabolic bundles}

Here we  describe the coarse moduli space $P_d(\bt)$  of undecomposable 
quasi-parabolic bundles $(E, \bl)$  of rank $2$ and of degree $d$ over
$(X,D)=(\pc, \bt) $.  When $\bnu$ is generic,  $P_d(\bt)$  can be also 
identified with  the coarse 
moduli space of $\bnu$-flat quasi-paraboli bundles of rank $2$ of degree $d$. 
Here a quasi-parabolic bundle $(E, \bl)$ is called {\em $\bnu$-flat}  if  
it  admit a connection $\nabla$ with given local exponents $\bnu$.

It is a non separated scheme constructed by patching together moduli spaces 
$P_d^{\bw}(\bt)$ of (semi)stable parabolic bundles 
for different choices of weights $\bw$ ;
those charts are smooth projective manifolds. A similar description
has been done in \cite{ArinkinLysenko1}.
In the degree $d=-1$ case, using Higgs fields and apparent map for a convenient cyclic vector, 
we define a natural ``birational'' map
$P_{-1}(\bt)\dashrightarrow \vert\cO_X(n-3)\vert^*\simeq\mathbb P_{\cc}^{n-3}$,
which turn to be an isomorphism in restriction to 
one of the projective charts $P_{-1}^{\bw}({\bt})$, 
for a convenient choice of weights. 
Under the condition that $\bnu$ is generic, 
this will be used in the next section to compute and describe the forgetful map 
\begin{equation}\label{eq:forgetful}
M(\bt,\bnu) \to P_{-1}(\bt) \ ;\ (E, \nabla,  \bl )\mapsto(E,\bl).
\end{equation}

\subsection{The flatness of undecomposable quasi-parabolic bundles}

In order to define the forgetful map (\ref{eq:forgetful}), we would like to 
characterize $\bnu$-flat quasi-parabolic bundles $(E, \bl)$ of rank $2$, 
which are,  by definition,  quasi-parabolic bundles $(E,\bl)$ of  rank $2$ and of degree $d$ 
on $(\pc, \bt)$ arising in our moduli spaces of parabolic connections 
$M(\bt, \boldsymbol{\nu})$,  i.e. admitting a connection $\nabla$ with prescribed poles, parabolics and eigenvalues. 
Under the condition that $\bnu$ is generic, this is given by the parabolic version of Weil criterium, see for instance in \cite[Proposition 3]{ArinkinLysenko1}.

\begin{proposition}\label{prop:flat=undecomposable}
Assume $\nu_1^{\epsilon_1}+\cdots+\nu_n^{\epsilon_n} \not \in 
\mathbb Z$ for any $\epsilon_i\in\{+,-\}$. 
Given a quasi-parabolic bundle $(E,\boldsymbol{l})$,  the following condition are equivalent:
\begin{enumerate}
\item $(E,\boldsymbol{l})$ is $\bnu$-flat, that is, 
$(E,\boldsymbol{l})$ 
admits a parabolic connection $\nabla$ with eigenvalues $\boldsymbol{\nu}$, 
\item $(E,\boldsymbol{l})$ is simple: the only automorphisms of $E$ preserving parabolics are scalar,
\item $(E,\boldsymbol{l})$ is undecomposable: there does not exist decomposition $E=L_1\oplus L_2$
such that each parabolic direction $l_i$ is contained either in $L_1$ or in $L_2$.
\end{enumerate}
\end{proposition}

\begin{remark}{\rm When $\nu_1^{\epsilon_1}+\cdots+\nu_n^{\epsilon_n} \in 
\mathbb Z$ for some $\epsilon_i\in\{+,-\}$, we still have 
[$\text{simple}\Leftrightarrow\text{undecomposable}\Rightarrow\text{flat}$]
but some decomposable parabolic bundles also admit connections compatible with $\bnu$.}
\end{remark}


We promptly deduce the following obstruction on $E$. 

\begin{corollary}Write $E=\mathcal O_{X}(d_1)\oplus \mathcal O_{X}(d_2)$ with
$d_1\le d_2$, $\deg(E)=d_1+d_2=d$.
Then $E$ admits an undecomposable quasi-parabolic structure if, 
and only if 
$$d_2-d_1\le n-2\ \ \ \text{(except $n=2$ and $d_1=d_2$ which is decomposable).}$$
\end{corollary}

\begin{proof}
When $d_1=d_2$, any decomposition of 
$E = \mathcal O_{X}(d_1)\oplus \mathcal O_{X}(d_1)$ 
is given by two distinct embeddings 
$\mathcal O_{X}(d_1)\hookrightarrow E$; one such embedding 
is determined once you ask it to contain one parabolic direction. Then for $n=2$ (or less)
we can decompose the parabolic data, while for $n\ge 3$, we can easily construct
a non decomposable parabolic structure.

When $d_1<d_2$, any decomposition of $E$ 
is given by the destabilizing bundle 
$L_2=\mathcal O_{X}(d_2)$
and any embeddings $\mathcal O_{X}(d_1)\simeq L_1\hookrightarrow E$. 
Latter ones form a family of dimension
$n'=d_2-d_1+1$: more precisely, $n'$ parabolics  lying outside of  
$\mathcal O_{X}(d_2)$ are always contained in such subbundle
and determine it.
If $n\le d_2-d_1+1$ then any quasi-parabolic structure is thus 
decomposable; if $n>d_2-d_1+1$, it suffices to choose
all $l_i$'s outside $\mathcal O_{X}(d_2)$, and $l_n$ outside 
of the $\mathcal O_{X}(d_1)$ defined by the $d_2-d_1+1$ first parabolics.
\end{proof}

One of the difficulty to define the forgetful map 
$(E,\nabla,\bl)\mapsto(E,\boldsymbol{l})$
is that, although the moduli space of connections is 
smooth and separated (for generic $\bnu$), the image is never separated.
The reason is that, although the former moduli spaces 
can be constructed as geometrical quotient of
stable objects, the latter one always contain unstable 
ones and will be not a good scheme.  

However, we have
\begin{proposition}\label{prop:undecomposable=stable}
A quasi-parabolic bundle $(E,\boldsymbol{l})$ is undecomposable if, and only if, it is stable
for a convenient choice of weights $\boldsymbol{w}$.
Hence, if we assume that $\bnu$ is generic, that is, 
$\nu_1^{\epsilon_1}+\cdots+\nu_n^{\epsilon_n} \not \in 
\mathbb Z$ for any $\epsilon_i\in\{+,-\}$, a quasi-parabolic bundle $(E,\boldsymbol{l})$
is $\bnu$-flat if and only if it is stable
for a convenient choice of weights $\boldsymbol{w}$.
\end{proposition}

This will allow us to construct our coarse moduli space of undecomposable quasi-parabolic bundles $P_d(\bt)$ by patching
together moduli spaces $P_d^{\bw}(\bt)$ where $\boldsymbol{w}$ runs over a 
finite family of convenient weights.  

\begin{proof}Let $(E,\boldsymbol{l})$ be a quasi-parabolic bundle and let us write $E=L_1\oplus L_2$ as above, 
$L_i=\mathcal O_{X}(d_i)$, $d_1\le d_2$, $d_1+d_2=d$. 

In the decomposable case, all parabolics are contained in the union $L_1\sqcup L_2$.
In this case, it is easy to check that $\mathrm{Stab}(L_1)+\mathrm{Stab}(L_2)=0$,
whatever the weights are, so that one of the two must be $\le0$:  
$(E,\boldsymbol{l})$ is not stable for any choice 
of weights. Note that it is however (strictly) semi-stable 
for a convenient choice of weights.

In the undecomposable case, we may choose $L_1$ passing through a maximum number of parabolics,
i.e. at least $d_2-d_1+1$. Choose $d_2-d_1$ of them and apply a negative elementary transformation at each of those directions.
We get a new undecomposable quasi-parabolic bundle $(E',\boldsymbol{l}')$ with $E'\simeq L_1\oplus L_1$. 
In particular, there are $3$ parabolics, say $l'_1$, $l'_2$ and $l'_3$, lying on $3$ distincts embeddings $L_1\subset E'$.
It is easy to check that this bundle is stable for weights $\boldsymbol{w}'$
$$0 <w_1'=w_2'=w_3'<\frac{2}{3}\ \ \ \text{and other}\ \ \ w_i'=0.$$
Thereore, the original quasi-parabolic bundle   
$(E,\boldsymbol{l})$ is stable for  weights $\boldsymbol{w}$ defined by 
\begin{itemize}
\item $w_j'=1-w_j$ if $l_j$ is one of the directions where we made elementary transformation,
\item $w_j'=w_j$ for other parabolics.
\end{itemize}
\end{proof}

\subsection{GIT moduli spaces of stable parabolic bundles}\label{sec:GITbundle}

Given weights, the moduli space of semistable points 
$P_d^{\bw}(\bt)$ is a (separated but may be singular) 
projective variety; moreover, stable points are smooth 
(cf. \cite{MehtaSeshadri}, \cite{Bhosle},  
\cite{BodenHu}, \cite{MaruyamaYokogawa}, \cite{BiswasHollaKumar} and \cite{Yokogawa}). 

Let $\mathcal W:=[0,1]^n$ be the set of weights. 
Given a parabolic bundle $(E,\boldsymbol{l})$
of degree $d$ and a  line subbundle $L\hookrightarrow E$ 
of degree $k$, denote by $I_1$ the set of indices of those parabolic 
directions contained in $L$, and $I_2 = \{1,\ldots,n\} \setminus I_1$, 
so that $\{1,\ldots,n\}=I_1\sqcup I_2$ .  
Then, the (parabolic) stability index of $L$ is zero if, and only if, the weights
lie along the hyperplane
$$
H_d(k,I_1):=\{\boldsymbol{w}\ ;\ d-2k-\sum_{i\in I_1} w_i+\sum_{i\in I_2} w_i=0\}.
$$
This equality cuts out the set of weights $[0,1]^n$ into two open sets 
(possibly one is empty, e.g. for large $k$).
Note that the wall $H_d(k,I_1)=H_d(d-k,I_2)$ also bound the stability locus
of those $L'=\mathcal O_{X}(d-k)\hookrightarrow E'$ 
passing through parabolics $I_2$ (with $\deg(E')=d$).
On one side those $L\hookrightarrow E$ are destabilizing, 
on the other side those $L'\hookrightarrow E'$
are destabilizing. Along the wall, decomposable parabolic bundles 
$\mathcal O_{X}(k)\oplus\mathcal O_{X}(d-k)$
with parabolics distributed on the two factors as 
$\{1,\ldots,n\}=I_1\sqcup I_2$ may occur as strictly 
semi-stable points in $P_d^{\bw}(\bt)$.

When we cut out $[0,1]^n$ by all possible walls $H_d(k,I_1)$ (only finitely
many intersect) we get in the complement many chambers (connected components) along which the moduli
space only consists of stable parabolic bundles (semi-stable $\Rightarrow$ stable). 
Thus $P_d^{\bw}(\bt)$ is smooth and locally constant along each chamber.

Mind that $P_d^{\bw}(\bt)$ may be empty over some chambers 
like it so happens for $k$ odd and $\boldsymbol{w}=(0,\ldots,0)$: the bundle is unstable
since it is in the usual sense (parabolics are not taken into account for vanishing weights).

A parabolic bundle $(E,\boldsymbol{l})$ is said to be {\em generic} if 
\begin{itemize}
\item $E=\mathcal O_{X}(d_1)\oplus\mathcal O_{X}(d_2)$ with $0\le d_2-d_1\le 1$,
\item a line bundle $L\subset E$ cannot contain more that $m+1$ parabolics, where $m=\deg(E)-2\deg(L)$. 
\end{itemize}
Note that $m+1$ is the dimension of deformation for such a subbundle $L\subset E$.
It is easy to see that the set $P_d(\bt)^0\subset P_d(\bt)$ of generic bundles $(E, \bl)$ over $(X,D)=(\pc, \bt)$ is a variety of 
dimension $n-3$.

\begin{proposition}A generic parabolic bundle $(E,\bl) \in P_d(\bt)^0$ is stable for  weights $\bw$ if, and only if, the weights $\boldsymbol{w}$
satisfy all inequalities
$$m-\sum_{i\in I_1}w_i+\sum_{i\in I_2}w_i>0$$
with $\#I_1=m+1$ and $m\ge0$ integer, $m\equiv d\ \text{mod}\ 2$.
The moduli space $P_d^{\bw}(\bt)$ has therefore expected
dimension $n-3$ and contains $P_d(\bt)^0$ as an open subset.
\end{proposition}

A weight $\boldsymbol{w}$ is said to be {\em admissible} if it lies outside the walls and moreover satisfies
 all above inequalities; a chamber is admissible if it consists of admissible weights.
It follows that the intersection of all $P_d^{\bw}(\bt)$ over admissible weights
have a common open subset of dimension $n-3$, namely the moduli space of generic bundles $P_d(\bt)^0$
(that does not depend on choice of generic weights).

Let  us denote by $\cW_{adm}$ the set of admissible weights and decompose it into 
the finite number of connected components $ \cW_{adm} = \cup_{i, finite} \cW_{adm, i}$ separated by  walls.   
Then the moduli space $P^{\bw}_d(\bt)$ is constant over the connected components $\bw \in \cW_{adm, i}$.  
So we may choose a representative $\bw_i \in \cW_{adm, i}$ for each $i$, and for different $i_1, i_2$, 
we have a big nonempty open set $U_{i_1, i_2} \supset P_d(\bt)^0$  such that 
$$
P_d^{\bw_{i_1}}(\bt) \supset U_{i_1, i_2} \subset P_d^{\bw_{i_2}}(\bt).
$$ 
We can patch all of $P_d^{\bw_{i}}(\bt)$  over these open subets 
$U_{i_1, i_2}$ and obtain the coarse moduli space of undecomposable parabolic bundles of degree $d$.  
Thus we have the following 

\begin{proposition}\label{prop:modulistack} The coarse moduli space $P_{d}(\bt)$ 
of undecomposable parabolic bundles of degree $d$ can be obtained as  
\begin{equation}\label{eq:patch}
P_{d}(\bt) = \cup_{i, finite } P_d^{\bw_{i}}(\bt).
\end{equation} 
\end{proposition}

Recall (cf. Proposition \ref{prop:flat=undecomposable}) that undecomposable $\Leftrightarrow$ flat for generic $\bnu$. 
However, for special $\bnu$,
only $\Rightarrow$ holds and there are decomposable flat bundles. The coarse moduli space of flat bundles might more 
difficult to describe then.

\subsection{Wall-crossing and non separated phenomena}\label{section:WallCrossing}

Let us compare the moduli spaces $P_d^{\bw}(\bt)$ corresponding
to admissible chambers, say $W^1$ and $W^2$, separated by a wall $H_d(k,I_1)$. 
Applying elementary transformations  if necessary,  we can assume $d=k=0$. Those bundles with  parabolics $I_1$ (resp. $I_2$)
on the same $\mathcal O_{X}\hookrightarrow E$ are excluded on $W^2$ (resp. on $W^1$).
Along the wall $H_0(0,I_1)=H_0(0,I_2)$, both are allowed as strictly semi-stable bundles;
they identify, in the moduli space, with the decomposable bundle $E=L_1\oplus L_2$ having 
parabolics $I_i$ on $L_i$. The special two kinds of bundles previously described yield non separated
points in the quotient space. Indeed, they are defined, on the trivial bundle
($L_1=L_2=\mathcal O_{X}$), by
\begin{itemize}
\item $ \boldsymbol{l}^1 $ spanned by $ \begin{pmatrix}1\\0\end{pmatrix} $ for $ i\in I_1$ and $\begin{pmatrix}u_i\\1\end{pmatrix}$ for $i\in I_2$,
\item $\boldsymbol{l}^2$ spanned by $\begin{pmatrix}1\\v_i\end{pmatrix}$ for $i\in I_1$ and $\begin{pmatrix}0\\1\end{pmatrix}$ for $i\in I_2$.
\end{itemize}
If $u_i$ (resp. $v_i$) are generic enough in $\mathbb C$, then $\boldsymbol{l}^i$ defines a stable parabolic
bundle on $W^i$ but is no more semi-stable for the other chamber $W^j$, $\{i,j\}=\{1,2\}$. Now, consider the  
one-parameter family of 
parabolic structures defined by 
\begin{itemize}
\item $\boldsymbol{l}^{\varepsilon}$ spanned by $\begin{pmatrix}1\\ \varepsilon v_i\end{pmatrix}$ for $i\in I_1$ and $\begin{pmatrix}\varepsilon u_i\\1\end{pmatrix}$ for $i\in I_2$.
\end{itemize}
When $\varepsilon \sim 0$, this parabolic structure is stable on both $W^i$'s and when $\varepsilon \to 0$,
it tends to either $\boldsymbol{l}^1$, or $\boldsymbol{l}^2$, depending on the chamber.

We can easily deduce any other non separating phenomenon applying back elementary transformations.
For instance, the wall $H_0(1,\emptyset)$ (i.e. $I_1=\emptyset$) is separating the locus of stability of
\begin{itemize}
\item those parabolic structure on the non trivial bundle $E=\mathcal O_{X}(-1)\oplus\mathcal O_{X}(1)$;
\item those parabolic structure on the trivial bundle $E=\mathcal O_{X}\oplus\mathcal O_{X}$ where all parabolics
lie along the same $\mathcal O_{X}(-1)\hookrightarrow E$.
\end{itemize}
This provides a non separated phenomenon: former parabolic bundles are arbitrary close to latter ones
and {\it vice-versa}. Indeed, after applying elementary transformation to, say, $l_1$ and $l_2$, we are
back to a special case of the above discussion.

\subsection{The coarse moduli space of undecomposable quasi-parabolic structures on the trivial bundle}\label{subsection:trivialbundle}

Now we describe (following and completing  \cite[Section 2.3]{ArinkinLysenko1}) the coarse moduli space $P_0$ of undecomposable quasi-parabolic bundles $(E,\boldsymbol{l})$ on $(\pc, \bt)$ of rank $2$ and of degree $0$. 
The cases of all even degrees are similar after twisting by a convenient line bundle.

As suggested by the proof of Proposition \ref{prop:undecomposable=stable}, the coarse moduli space $P_0(\bt)$   is covered by open charts of the following type.

For $\{i,j,k\}\subset\{1,\ldots,n\}$, consider the moduli space of stable parabolic bundles of degree $0$
with respect to weights $\boldsymbol{w}$ defined by 
$$0<w_i=w_j=w_k<\frac{2}{3}\ \ \ \text{and other}\ \ \ w_l=0.$$
Such parabolic bundles are exactly given by those parabolic structures on the trivial bundle 
$E=\mathcal O_{X}\oplus\mathcal O_{X}$ 
such that $l_i$, $l_j$ and $l_k$ are pairwise distinct (through the trivialization of $E$).
Indeed, it cannot be $E=\mathcal O_{X}(-1)\oplus\mathcal O_{X}(1)$ for instance, since in this case
$\mathcal O_{X}(1)$ is destabilizing (taking weights into account). Also, on $E=\mathcal O_{X}\oplus\mathcal O_{X}$, 
$l_i\not=l_j$ otherwise the trivial line bundle $\mathcal O_{X}\hookrightarrow E$ that contains these
directions would be destabilizing. Here we get a fine moduli space that can be described as follows.
Choose a trivialization $\mathbb C^2$ of $E$ such that $l_i=(1:0)$, $l_j=(1:1)$ and $l_k=(0:1)$.
Then our moduli space identifies with
$$U_{i,j,k}=\left\{\boldsymbol{l}\ ;\ 
\begin{matrix}l_i=(1:0),\\ l_j=(1:1),\\ l_k=(0:1)\end{matrix}\ \ \ \text{and other}\ l_l\in\pc\ \text{arbitrary}\right\}\simeq\left(\pc\right)^{n-3}.
$$

Let $P_{0,0}(\bt)$ denote the moduli subspace of $P_{0}(\bt)$ of  
$(E, \bl)$ with the trivial vector bundle $E \simeq \Oo \oplus \Oo$.   
To get all points of $P_{0,0}(\bt)$, we have to patch 
all these projective smooth charts $U_{i,j,k}$ together: any two of them intersect on a non empty open subset.
We already obtain a non separated scheme. For $n=4$, we obtain $\pc$ with $3$ double points 
at $0$, $1$ and $\infty$, or equivalently, two copies of $\pc$ glued outside $0$, $1$ and $\infty$.
For $n=5$, gluing maps are birational and non separated phenomena increase: there are rational curves 
arbitrary close to points. For $n=6$, there are rational curves arbitrary close to each other through a flop.

Usually, the GIT compactification $\overline{M(0,n)}$ of the moduli of $n$-punctured sphere is constructed by setting all $w_i=1/n$. When $n$ is even, this weight is along 
the walls $H(0,I_1)$ with $\# I_1=\frac{n}{2}$ and there are strictly semi-stable (and decomposable) bundles.
On the other hand, when $n$ is odd, the weight is inside a chamber. For instance, for $n=5$, we get the $3$ blow-up of $\pc\times\pc$ at the points $(0,0)$, $(1,1)$ and $(\infty,\infty)$. Although this latter moduli space does not embed
in any chart $U_{i,j,k}\simeq \pc\times\pc$ considered above, it embeds in the total coarse moduli space 
$P_{0}(\bt)$ as an open subset.

\subsection{The coarse moduli space of undecomposable quasi-parabolic structures on degree $d$ bundles}

All charts $U_{i,j,k}$ are not enough to cover all quasi-parabolic bundles $(E,\boldsymbol{l})$ of degree $0$: we only get
those for which $E$ is the trivial bundle. We have to add other charts that can be deduced from previous ones
by making any even number of elementary transformations (and twisting by the convenient line bundle).
All in all, it is enough to consider the following set of weights
\begin{equation}\label{ExhaustiveWeights}
W:=\left\{\boldsymbol{w}\ ;\ \ 
\begin{matrix}$3$  \text{ of } w_i\text{'s are } \frac{1}{2}\\ \text{other } w_i\text{'s are } 0 \text{ or } 1\end{matrix}\right\}
\end{equation}
and the corresponding moduli spaces, all isomorphic to $\left(\pc\right)^{n-3}$. In fact, those $\boldsymbol{w}\in W$ for which $1$ does not occur are exactly those charts $U_{i,j,k}$ above; other ones are deduced by even numbers
of elementary transformations. Recall that stability and flatness are invariant under elementary transformations.

Let us set
$$
P_{d, k}(\bt) = \{ (E, \bl) \in P_{d}(\bt) \ | \ E \simeq \Oo(k) \oplus \Oo(d-k) \}
$$
Then we have a stratification $P_{0}(\bt)=P_{0, 0}(\bt) \sqcup P_{0,1}(\bt) \sqcup\cdots\sqcup P_{0, m}(\bt)$, where $m$ is positive integer 
maximal such that $m \le \frac{n-2}{2}$. All $P_{d, k}(\bt)$ with $k>0$ are on the non separated locus of $P_{0}(\bt)$. The open separated locus $P_{0}(\bt)^0$ (generic bundles) is, inside $P_{0,0}(\bt)$,
the complement of those parabolic structures for which a  
subbundle $L\hookrightarrow \mathcal O_{X}\oplus\mathcal O_{X}$ passes through an exceeding number of parabolics.

From the consideration as above, we can see that in the patching (\ref{eq:patch}) $P_{0}(\bt) = \cup_{i} P_0^{\bw_i} (\bt)$,  the charts
 $P_0^{\bw_i} (\bt)$ with $\bw_i \in W$ given by (\ref{ExhaustiveWeights}) are enough to cover the whole coarse moduli space.

We can promtly deduce the coarse moduli space $P_{-1}(\bt)$ of quasi-parabolic bundles  $(E,\boldsymbol{l})$
of degree $-1$  from the previous discussion by applying a single elementary transformation, 
say at $t_n$
$$\mathrm{Elm}_{t_n}^-:P_{0}(\bt)\stackrel{\sim}{\lra} P_{-1}(\bt).$$
We get a stratification 
$P_{-1}(\bt)=P_{-1, 0}(\bt) \sqcup P_{-1, 1}(\bt) \sqcup\cdots\sqcup P_{-1, m}(\bt)$, 
where $m$ is  maximal such that  $m \le \frac{n-3}{2}$. 

\subsection{A natural projective chart for coarse moduli space of degree $-1$ bundles}\label{subsection:chartV}

A natural projective chart $V$ is given by those undecomposable parabolic structures on $E=\mathcal O_{X}\oplus\mathcal O_{X}(-1)$ where no parabolic lie on $\mathcal O_{X}$.

\begin{proposition}\label{Prop:OurMainChart}
Assume $n\ge 3$.
For ``democratic'' weights $w_i=w$, $i=1,\ldots,n$, with $\frac{1}{n}<w<\frac{1}{n-2}$,
a degree $-1$ parabolic bundle $(E,\boldsymbol{l})$ is (semi-)stable if, and only if
\begin{itemize}
\item  $E=\mathcal O_{X}\oplus\mathcal O_{X}(-1)$,
\item no parabolic $l_i$ lie on $\mathcal O_{X}$,
\item not all $l_i$ lie on the same $\mathcal O_{X}(-1)$ (flatness).
\end{itemize}symmetry
Moreover, for these weights, the moduli space $V:=P_{-1}^{\bw}(\bt)$ is naturally isomorphic to 
$\mathbb PH^0(X,L\otimes\Omega_X^1(D))^*\simeq\pp_{\cc}^{n-3}$, where $L=\mathcal O_{X}(-1)$.
\end{proposition}

\begin{proof}
That $\mathcal O_{X}$ free of parabolics does not destabilize the parabolic bundle is equivalent to $\frac{1}{n}<w$.
On the other hand, for $w<\frac{1}{n-2}$, a $\mathcal O_{X}(-1)$ passing through $n-1$ parabolics does not
destabilize, but the parabolic bundle becomes unstable whenever one parabolic lie on $\mathcal O_{X}$.
Finally, to eliminate parabolic structures on degree $-1$ vector bundles $E\not=\mathcal O_{X}\oplus\mathcal O_{X}(-1)$, 
we just need $w<\frac{3}{n}$ which is already implied by the above inequalities provided that  $n\ge 3$.

Parabolic bundles of the chart $V$ are precisely non trivial \emph{extensions}
$$0\to(\mathcal O_X,\emptyset)\to (E,\bl)\to (L,D)\to 0$$
which means that the pair is defined by gluing local models $U_i\times\mathbb C^2$
(for an open analytic covering $(U_k)$ of $X$) by transition matrices 
$$M_{kl}=\begin{pmatrix}1 & b_{kl}\\ 0 & a_{kl}\end{pmatrix}$$
with $b_{kl}$ vanishing on $D$. Here, on each chart $U_k$, the vector $e_1$ 
generates the trivial subbundle $\mathcal O_X\hookrightarrow E$ and $e_2$ 
gives the parabolic direction over each point of $D$.

The multiplicative cocycle $(a_{kl})_{kl}\in H^1(X,\mathcal O_X^*)$ defines the line bundle $L$.
Let $a_{kl}=\frac{a_k}{a_l}$ be a meromorphic resolution: $a_i$ is meromorphic on $U_k$ with $\mathrm{div}(a_k)=\mathrm{div}(L)$.
The obstruction to split the extension is measured by an element of 
$$H^1(X, \mathrm{Hom}(L(D),\mathcal O_X))=H^1(X,L^{-1}(-D))=H^0(X,L\otimes\Omega_X^1(D))^*$$
(by Serre duality) which is explicitely given by $(b_{kl}a_l)_{kl}\in H^1(X,L^{-1}(-D))$. Any two non trivial extensions define isomorphic
parabolic bundles if and only if the corresponding cocycles are proportional: the moduli space
of extensions is parametrized by $\mathbb PH^0(X,L\otimes\Omega_X^1(D))^*$.
\end{proof}

\subsection{Case $n=4$ detailled}\label{subsec:n=4detailled}

For degree $0$ and undecomposable parabolic bundles, we have the following possibilities:
\begin{itemize}
\item $E$ is the trivial bundle and at most two of the $l_i$'s coincide;
\item $E=\mathcal O_{X}(-1)\oplus\mathcal O_{X}(1)$ and in this case, there is a unique
undecomposable quasi-parabolic structure up to automorphism, say $l_1,l_2,l_3\in\mathcal O_{X}(-1)$
and $l_4$ outside of the two factors.
\end{itemize}
On the space $[0,1]^4$ of weights, the walls are defined by equations of the type 
$$\epsilon_1 w_1+\cdots+\epsilon_4 w_4\in2\mathbb Z$$
where $\epsilon_i=\pm$ and we get the following possibilities (other ones do not cut out $[0,1]^4$ into two non empty pieces)
$$\begin{matrix}
w_1+w_2+w_3+w_4&=&2\\
w_i+w_j+w_k-w_l&=&0\text{ or }2\\
w_i+w_j-w_k-w_l&=&0
\end{matrix}$$
where $\{i,j,k,l\}=\{1,2,3,4\}$, which gives $1+2\cdot4+3=12$ walls. It is easy to check that the moduli space of
(semi-)stable parabolic bundles is non empty if, and only if, we have the following inequalities
$$0\le w_i+w_j+w_k-w_l\le 2.$$
For instance, when $w_1+w_2+w_3<w_4$, then the line bundle $\mathcal O_{X}\hookrightarrow E$ passing through
$l_4$ destabilizes the bundle.

Now, under above inequalities, the remaining $4$ walls cut out the remaining space of weights into $16$ chambers.
For $w_1+w_2+w_3+w_4<2$, the moduli space $P_0^{\bw}(\bt)$ consists only of parabolic structures 
on the trivial bundle: $\mathcal O_{X}(1)\hookrightarrow\mathcal O_{X}(-1)\oplus\mathcal O_{X}(1)$ is destabilizing in this case.
This half-space splits into $8$ admissible chambers, but only $4$ are enough to cover all quasi-parabolic structures on $\mathcal O_{X}\oplus\mathcal O_{X}$, namely those containing $\boldsymbol{w}_4=(\frac{1}{2},\frac{1}{2},\frac{1}{2},0)$, and its permutations $\boldsymbol{w}_i$ (the $i$th weight is zero).
For $\boldsymbol{w}_4$, we get the following chart
$$U_{1,2,3}:=\{\boldsymbol{l}=(0,1,\infty,u)\ ;\ u\in\pc\}.$$
The classical moduli space $M(0,4)$ is given by the open set $u\not=0,1,\infty$
and this is the locus $P_0(\bt)^0$ of generic parabolic bundles.
The chart given by $\boldsymbol{w}_1$ can be for instance described as
$$U_{4,2,3}:=\{\boldsymbol{l}=(v,1,\infty,0)\ ;\ v\in\pc\}.$$
The intersection is given, in $U_{1,2,3}$, by the complement of $l_4=l_2$ and $l_4=l_3$,
i.e. by $u\not=1,\infty$. The two projective charts glue along the latter open subset through
the fractional linear transformation $U_{1,2,3}\to U_{4,2,3};u\mapsto v=\frac{u}{u-1}$.
We have already added two non separated points, namely at $u=1$ and $u=\infty$. 

After patching all $4$ charts together, we get a non separated scheme over $\pc\ni u$
with double points over $u=0$, $1$ and $\infty$; they correspond to pairs
of special parabolic structures respectively defined by
$$\{l_1=l_4\text{ or }l_2=l_3\},\ \ \ \{l_2=l_4\text{ or }l_1=l_3\}\ \ \ \text{and}\ \ \ \{l_3=l_4\text{ or }l_1=l_2\}.$$
Finally, one has to add the unique undecomposable quasi-parabolic structure on the non trivial
bundle $\mathcal O_{X}(1)\oplus\mathcal O_{X}(-1)$. This adds a $4$th non separated point over $u=t$
where $\text{cross-ratio}(0,1,\infty,t)=\text{cross-ratio}(t_1,t_2,t_3,t_4)$. Indeed, it is
infinitesimally closed to the (unique) quasi-parabolic structure lying on an embedding $\mathcal O_{X}(-1)\hookrightarrow\mathcal O_{X}\oplus\mathcal O_{X}$.

To end with the degree $0$ case, we note that, although the coarse moduli space is constructed {\it a posteriori} by gluing 
two copies of $\pc$ along the complement of $t_1,t_2,t_3,t_4$. However, this identification strongly
depend on our choice of the initial chart $U_{1,2,3}\in u$. Starting from another chart will give another
identification; this is up to the $4$-group that preserves the cross-ratio.

In case of degree $-1$ bundles, we necessarily have $E=\mathcal O_{X}\oplus\mathcal O_{X}(-1)$ by undecomposability.
Let us choose weights $w_1=w_2=w_3=w_4=:w$. The moduli space $P_{-1}^{\bw}(\bt)$
is non empty for $\frac{1}{4}\le w\le \frac{3}{4}$. At $w=\frac{1}{2}$ only, strictly semistable bundles occur. 
There are two chambers, namely
\begin{itemize}
\item $\frac{1}{4}< w< \frac{1}{2}$ where no parabolic $l_i$ is contained in $\mathcal O_{X}$;
\item $\frac{1}{2}< w< \frac{3}{4}$ where not $3$ of the $l_i$'s is contained in the same $\mathcal O_{X}(-1)$.
\end{itemize}
By this way, the coarse moduli space is constructed by only two open projective charts, and the four double points
are given by those pairs
$$\{l_i\text{ is contained in }\mathcal O_{X}\}\ \ \text{  and  }\ \ \{l_j,l_k,l_l\text{ are contained in the same }\mathcal O_{X}(-1)\}$$
which naturally identify with $t_i$.
Here, we get a natural identification with two copies of $\mathbb P^1$ glued along the complement of $t_1,t_2,t_3,t_4$.

\section{The two Lagrangian fibrations}

\subsection{Moduli of generic connections}
All along this section, we fix ``democratic''  weights 
\begin{equation}\label{eq:democratic}
\bw=(w,\ldots,w) \ \ \mbox{with} \ \ \frac{1}{n}<w<\frac{1}{n-2}
\end{equation}
like in Proposition \ref{Prop:OurMainChart} and we consider the moduli space 
$M^{\bw}(\bt,\bnu)$ of $\bw$-stable $\bnu$-parabolic connections $(E,\nabla,\bl)$ 
where $\bnu=(\nu_i^{\pm})$ with $\sum_{i}(\nu_i^++\nu_i^-)=1$
(see Section \ref{Sect:ModuliConnection}). Denote by $L=\mathcal O_{X}(-1)$ the determinant
line bundle. 
By Proposition \ref{Prop:OurMainChart},  
for the weights $\bw=(w,\ldots,w)$ in (\ref{eq:democratic}), 
the coarse moduli space  $V=P^{\bw}_{-1}(\bt)$ of $\bw$-stable parabolc bundles of degree $-1$ 
is isomorphic to 
$PH^0(X,L\otimes\Omega_X^1(D))^*\simeq\pp_{\cc}^{n-3}$ and consists of $(E, \bl)$ 
satisfying the conditions:
\begin{itemize}
\item $E=\mathcal O_{X}\oplus\mathcal O_{X}(-1)$,
\item $l_i\not \subset \mathcal O_{X}$ for $i=1,\ldots,n$,
\item not all $l_i$ lie in the same $\mathcal O_X(-1)\hookrightarrow E$.
\end{itemize}
Now we introduce the following  open subset of the moduli space $M^{\bw}(\bt,\bnu)$
\begin{definition}\label{def:goodmoduli}{\rm  For the weight $\bw$ in (\ref{eq:democratic}), 
let us define the open subset 
\begin{equation}\label{eq:generic-con}
M^{\bw}(\bt,\bnu)^0 = \{ (E, \nabla, \bl) \in M^{\bw}(\bt,\bnu) \ | \  (E, \bl) \in P^{\bw}_{-1}(\bt) \}
\end{equation}
of $ M^{\bw}(\bt,\bnu)$, which we call the moduli space of {\em generic $\bnu$-parabolic connections}.}
\end{definition}
We can define two natural Lagrangian maps on $M^{\bw}(\bt,\bnu)^0$. The first one 
\begin{equation}\label{eq:app}
\mathrm{App}:M^{\bw}(\bt,\bnu)^0\to \mathbb PH^0(X,L\otimes\Omega_X^1(D))\simeq\vert\mathcal O_{X}(n-3)\vert \simeq \pp_{\ba}^{n-3}
\end{equation}
is obtained by taking the apparent singular points with respect to the cyclic vector 
of the global section of $\mathcal O_{X}$.
Precisely, each connection $\nabla$ on $E=\mathcal O_{X}\oplus\mathcal O_{X}(-1)$ defines
a $\mathcal O_{X}$-linear map
$$\mathcal O_{X}\hookrightarrow E\stackrel{\nabla}{\longrightarrow} E\otimes \Omega^1_{X}(D)\to 
(E/\mathcal O_X)\otimes\Omega_X^1(D)\simeq L\otimes \Omega^1_{X}(D)$$
(where the last arrow is the quotient by the subbundle defined by $\mathcal O_{X}\hookrightarrow E$) i.e. a map
$$\varphi_{\nabla}:\mathcal O_{X}\to L\otimes\Omega^1_{X}(D).$$
Its zero divisor is an element of the linear system $\mathbb PH^0(X,L\otimes\Omega_X^1(D))\simeq \vert\mathcal O_{X}(n-3)\vert$.
This map extends as a rational map 
$$\mathrm{App}:M^{\bw}(\bt,\bnu)\dashrightarrow \vert\mathcal O_{X}(n-3)\vert$$
on the whole moduli space with some indeterminacy points  (See \cite{SaitoSzabo}). 

The second Lagrangian map 
\begin{equation}\label{eq:bun}
\mathrm{Bun}:M^{\bw}(\bt,\bnu)^0\to P^{\bw}_{-1}(\bt) \simeq\mathbb PH^0(X,L\otimes\Omega_X^1(D))^* \simeq 
(\pp_{\ba}^{n-3})^{\ast} \simeq \pp_{\bb}^{n-3}.
\end{equation}
comes from the forgetfull map towards the coarse moduli space of undecomposable 
 parabolic bundles
$$\mathrm{Bun}:M^{\bw}(\bt,\bnu)\to P_{-1}(\bt)\ ;\ (E,\nabla,\bl)\mapsto (E,\boldsymbol{l})$$
that we restrict to the open projective chart $V:=P^{\bw}_{-1}(\bt)$
of Section $\ref{subsection:chartV}$.

One of main results of this section is the following

\begin{theorem}\label{thm:embedding}
Under the assumption that $\sum_i \nu_i^-\not=0$ ($\Leftrightarrow\sum_i \nu_i^+\not=1$), the morphism 
\begin{equation}
\mathrm{App}\times\mathrm{Bun}\ :\ M^{\bw}(\bt,\bnu)^0\ \to\ \vert\mathcal O_{X}(n-3)\vert\times\vert\mathcal O_{X}(n-3)\vert^*
\simeq \pp_{\ba}^{n-3} \times \pp_{\bb}^{n-3}
\end{equation}
is an embedding. Precisely, the image is the complement of the incidence variety 
$\Sigma$ for the above duality.
\end{theorem}

\subsection{Compactification of the moduli space}\label{sec:compactmoduli}

In order to prove Theorem \ref{thm:embedding},  
we introduce a nice compactification $\overline{M^{\bw}(\bt, \bnu)^{0}}$ of the moduli space  
$M^{\bw}(\bt,\bnu)^0$ of generic connections 
and will show that  the extended map $\mathrm{App}\times\mathrm{Bun}$ to  $\overline{M^{\bw}(\bt, \bnu)^{0}}$ 
is in fact an isomorphism.  

In [Definition 2, \cite{Arinkin}], 
the moduli stack $\overline{\Mm (\bt, \bnu)}$ of 
$\lambda$-$\bnu$-parabolic connections $(E, \nabla, \varphi,  \lambda 
\in \cc, \bl)$ over $X = \pp^{1}$ 
are introduced. (Note that in \cite{Arinkin}, $\lambda$-$\bnu$-parabolic 
connections are called as  $\epsilon$-bundles.) 
Then under the conditions that $(E, \nabla)$ is irreducible, Arinkin   ([Theorem 1 in  
\cite{Arinkin}]) showed that the moduli stack  $\overline{\Mm(\bt, \bnu)}$ is a 
complete smooth Deligne-Mumford stack.  Moreover he also showed that 
the $\lambda=0$ locus $\overline{\Mm}(\bt, \bnu)_{H} \subset \overline{\Mm(\bt, \bnu)} $, 
which is the moduli stack of parabolic Higgs bundles, is  also a smooth 
algebraic stack. On the other hand, as remarked in the proof of [Proposition 7, \cite{Arinkin}], 
the coarse moduli space $\overline{M(\bt, \bnu)}$ corresponding to $\overline{\Mm(\bt, \bnu)}$is not smooth: it has  quotient singularities.  
(As for the possible smooth compactification by $\phi$-parabolic connections, one may refer  
\cite{InabaIwasakiSaito} and \cite{InabaIwasakiSaito2} (cf. Remark \ref{rem:smoothness}).)   
 
Our main strategy is to consider the coarse moduli space of 
$\bw$-stable  $\lambda$-$\bnu$-parabolic connections for the democratic weight $\bw$. Define 
the coarse moduli space   
\begin{equation}\label{eq:compactification}
\overline{M^{\bw}(\bt, \bnu)^{0}} := \left\{ 
\begin{array}{l}(E, \nabla, \lambda \in \cc, \bl) \\ 
\mbox{$\lambda$-$\bnu$-parabolic connection} 
\end{array}
\ | \ \mbox{ $ (E, \bl) \in  
 P^{\bw}_{-1}(\bt)$ }  \right\}/ \simeq . 
\end{equation}
Note that if  
$(E, \bl) \in P^{\bw}_{-1}(\bt)$, $\lambda$-parabolic connections 
$(E, \nabla, \lambda \in \cc, \bl)$ are always $\bw$-stable, and 
there exists a natural embedding $ M^{\bw}(\bt, \bnu)^{0} \subset \overline{M^{\bw}(\bt, \bnu)^{0}}$ 
such that 
\begin{equation}\label{eq:higgs-locus}
M^{\bw}(\bt, \bnu)_H^0:= \overline{M^{\bw}(\bt,\bnu)^0}\setminus M^{\bw}(\bt,\bnu)^0
\end{equation}
is the coarse moduli space of  parabolic Higgs bundles $(E, \nabla, 0 \in \cc, \bl)$ such that 
$(E, \bl) \in P^{\bw}_{-1}(\bt)$.  

We can describe the moduli space $\overline{M^{\bw}(\bt,\bnu)^0}$ 
naively as follows.
 Thinking of $\mathrm{Bun}:M^{\bw}(\bt,\bnu)^0\to P^{\bw}_{-1}(\bt)$ as an affine $\mathbb A^{n-3}$-bundle
over the projective chart $P^{\bw}_{-1}(\bt)$. 
On each parabolic bundle $(E,\bl)\in P^{\bw}_{-1}(\bt)$, any two connections $\nabla_0,\nabla_1$ compatible with $\bl$
differ to each other by a parabolic Higgs field 
$$\nabla_1-\nabla_0=\Theta\in H^0(\mathrm{End}(E,\boldsymbol{l})\otimes \Omega^1_{X}(D))$$
(residues of $\Theta$ are nilpotent on each fiber $E_{\vert t_i}$ fixing the parabolic direction $l_i$).  
The moduli space of connections identifies with the $(n-3)$-dimensional affine space 
$\nabla_0+H^0(\mathrm{End}(E,\boldsymbol{l})\otimes \Omega^1_{X}(D))$
(recall that $(E,\bl)$ is simple).  Let us consider the fiber 
$\mathrm{Bun}^{-1}(E, \bl)$ of the map $\mathrm{Bun}:M^{\bw}(\bt,\bnu)^0\to P^{\bw}_{-1}(\bt) $ in (\ref{eq:bun}) over $(E, \bl)$.  
A natural compactification of the fiber $\mathrm{Bun}^{-1}(E, \bl)$  is given by 
$$\overline{\mathrm{Bun}^{-1}(E, \bl)}:=
\mathbb P\left( \cc\cdot\nabla_0\oplus H^0(\mathrm{End}(E,\boldsymbol{l})\otimes \Omega^1_{X}(D)) \right).
$$
An element $\nabla:=\lambda \cdot\nabla_0+\Theta$ is a $\lambda$-connection; if $\lambda\not=0$,
it is homothetic equivalent to a unique connection, namely $\frac{1}{\lambda}\nabla$; if $\lambda =0$, 
it is a parabolic Higgs field. 
By this way, we compactify the fiber $\mathrm{Bun}^{-1}(E, \boldsymbol{l})$ by adding
$\mathbb PH^0(\mathrm{End}(E,\boldsymbol{l})\otimes \Omega^1_{X}(D))$.  
Varying $(E, \bl) \in  P^{\bw}_{-1}(\bt) $ and choose a local section $\nabla_0$ over  local 
open sets of $P^{\bw}_{-1}(\bt) $, 
we can construct a $\mathbb P^{n-3}$-bundle 
\begin{equation}\label{eq:c-bun}
\mathrm{Bun}:\overline{M^{\bw}(\bt,\bnu)^0} \lra V=P^{\bw}_{-1}(\bt)
\end{equation}
and its restriction to the boundary
\begin{equation}\label{eq:bd-bun}
\mathrm{Bun}_H: M^{\bw}(\bt, \bnu)_H^0:= (\overline{M^{\bw}(\bt,\bnu)^0}\setminus M^{\bw}(\bt,\bnu)^0)\lra V
\end{equation}
naturally identifies with the total space of the projectivized cotangent bundle $\mathbb PT^*V\to V$.


\subsection{Main Theorem}
The apparent map naturally extends on the compactification since $\varphi_\nabla$ 
can be defined in the same way for $\lambda$-connections (an Higgs fields).

Our main result, which will give a proof of Theorem \ref{thm:embedding}, now reads

\begin{theorem}\label{th:Main} We  fix the democratic weight $\bw=(w, \ldots, w)$ with 
$\frac{1}{n} < w < \frac{1}{n-2}$ and consider the moduli space $\overline{M^{\bw}(\bt, \bnu)^0}$ as in 
(\ref{eq:compactification}).  
If $\sum_i \nu_i^-\not=0$, the moduli space $\overline{M^{\bw}(\bt, \bnu)^0}$ is a smooth 
projective variety and 
the map $\mathrm{App}\times\mathrm{Bun}$ induces an isomorphism
\begin{equation}\label{eq:isomorphism}
\mathrm{App}\times\mathrm{Bun}:\overline{M^{\bw}(\bt,\bnu)^0}\stackrel{\sim}{\lra}\mathbb PH^0(X,L\otimes\Omega_X^1(D))\times\mathbb PH^0(X,L\otimes\Omega_X^1(D))^*.
\end{equation}
Moreover, by restriction, we also obtain the isomorphism 
\begin{equation}\label{eq:isom-boundary}
\mathrm{App}\times\mathrm{Bun}_{|M^{\bw}(\bt, \bnu)_H^0} : M^{\bw}(\bt, \bnu)_H^0\stackrel{\sim}{\lra} \Sigma
\end{equation}
where $\Sigma$  is the incidence variety for the duality.  
\end{theorem}

\begin{proof} It is enough to show that the natural morphism  
$\mathrm{App}\times\mathrm{Bun}$ (\ref{eq:isomorphism})  induces 
a regular isomorphism between algebraic varieties.  
Like in the proof of Proposition \ref{Prop:OurMainChart},
we consider a parabolic bundle $(E,\bl)\in P_{-1}^{\bw}(\bt)$ defined as an extension class, 
i.e. by a matrix cocycle 
$$M_{kl}=\begin{pmatrix}1 & b_{kl}\\ 0 & a_{kl}\end{pmatrix}$$
where the multiplicative cocycle $(a_{kl})$ defines the line bundle $L$
and the extension is equivalently defined by 
the cocycle $(b_{kl}a_l)_{kl}\in H^1(X,L^{-1}(-D))\simeq H^0(X,L\otimes\Omega_X^1(D))^*$
where $a_{kl}=\frac{a_k}{a_l}$ is a meromorphic resolution ($\mathrm{div}(a_k)=\mathrm{div}(L)$).
Let us fix also a non zero element $\gamma\in H^0(X,L\otimes\Omega_X^1(D))\setminus\{0\}$. 
We want to show that there is a unique 
$\lambda\in \mathbb C$ and a unique $\lambda$-connection $\nabla:E\to E\otimes\Omega_X^1(D)$ 
(compatible with $\bnu$ and $\bl$) realizing $\gamma$ as the apparent map.

Such a $\lambda$-connection $\nabla$ is given in charts $U_k$ by $\nabla=\lambda d+A_k$
\begin{equation}\label{eq:conn-matrix}
A_k=\begin{pmatrix} \alpha_k & \beta_k\\ \gamma_k & \delta_k\end{pmatrix}\in\mathrm{GL}_2(\Omega_{U_k}^1(D))
\end{equation}
with compatibility condition
\begin{equation}\label{CompatibilityConnection}
\lambda\cdot dM_{kl}+ A_kM_{kl}=M_{kl}A_l
\end{equation}
on each intersection $U_k\cap U_l$.
For each pole $z=t_i$, the residue of $A_k$ takes the form
\begin{equation}\label{ResidueCondition}
\mathrm{Res}_{t_i}(A_k)=\begin{pmatrix} \lambda\nu_i^- & 0\\ * & \lambda\nu_i^+\end{pmatrix}.
\end{equation}

The trace connection $\zeta$ is defined on $U_k$ by $d+\omega_k$ with compatibility conditions
$\frac{da_{kl}}{a_{kl}}+\omega_k-\omega_k=0$ on $U_k\cap U_l$. We must have 
\begin{equation}\label{TraceCondition}
\alpha_k+\delta_k=\lambda\omega_k
\end{equation}
on $U_k$. We note that $\frac{da_{kl}}{a_{kl}}=\frac{da_k}{a_k}-\frac{da_l}{a_l}$ so that 
$(\omega_k+\frac{da_k}{a_k})_k$ defines a global $1$-form, say $\omega\in H^0(X,L\otimes\Omega_X^1(D))$, and thus 
$\omega_k=\omega-\frac{da_k}{a_k}$.

Now, compatibility conditions \ref{CompatibilityConnection} expand as
\begin{equation}\label{CompConExpand}
\left\{\begin{matrix}
\frac{\gamma_k}{a_k}-\frac{\gamma_l}{a_l}&=&0& (\rightarrow\gamma:=\frac{\gamma_k}{a_k}=\frac{\gamma_l}{a_l}) \\
\alpha_k-\alpha_l&=&(b_{kl}a_l)\gamma&\\
\delta_k-\delta_l&=&-(b_{kl}a_l)\gamma-\lambda\frac{da_{kl}}{a_{kl}}&\\
a_k\beta_k-a_l\beta_l&=&-(\lambda a_ldb_{kl}+(b_{kl}a_l)(\alpha_k-\delta_l))&
\end{matrix}\right.
\end{equation}
The first condition says that all $(\frac{\gamma_k}{a_k})_k$ glue together to form a global section
$\gamma\in H^0(X,L\otimes\Omega_X^1(D))$. It defines the image of the apparent map.

Our problem is now: given $(b_{kl}a_l)_{kl}\in H^1(X,L^{-1}(-D))$ defining the parabolic bundle 
and $\gamma\in H^0(X,L\otimes\Omega_X^1(D))\setminus\{0\}$ defining the apparent data, prove that the 
matrix connections $\lambda d+A_k$ can be completed in a unique way, with a unique $\lambda$.

{\bf Step1: finding $\gamma_k$.}
Given $\gamma$, we obviously define $\gamma_k:=a_k\gamma\in H^0(U_k,\Omega_X^1(D))$.

{\bf Step2: finding $\alpha_k$.}
Fix $\alpha_k^0$ sections of $\Omega_X^1(D)$ on each $U_k$ realizing the residual data
$\mathrm{Res}_{t_i}(\alpha^0_k)=\nu_i^-$. The cocycle $(\alpha^0_k-\alpha^0_l)$ defines
an element of $H^1(X,\Omega_X^1)$ which is non zero: indeed, if we were able to solve the cocycle
by $\alpha_k^0-\alpha_l^0=\tilde\alpha_k-\tilde\alpha_l$ for some holomorphic $1$-forms $\tilde\alpha_k$,
then $(\alpha_k^0-\tilde\alpha_k)_k$ would define a global meromorphic $1$-form whose sum of residue
 $\sum_{i=1,\ldots,n}\nu_i^-\not=0$ contradicts Residue Theorem.
Now, we want to find $\alpha_k$ of the form $\lambda\alpha_k^0+\tilde\alpha_k$ with $\tilde\alpha_k$ holomorphic.
This means that we have to solve 
$$\tilde\alpha_k-\tilde\alpha_l=(b_{kl}a_l)\gamma - \lambda(\alpha_k^0-\alpha_l^0)$$
in $H^1(X,\Omega_X^1)$; since this cohomology group is one dimensional, $(\alpha_k^0-\alpha_l^0)$
is a generator (being non zero) and there is a unique $\lambda$ such that the right-hand-side is zero
in cohomology, providing solutions $(\tilde\alpha_k)$.
We have now fixed $\lambda$, and the solution $\alpha_k=\lambda\alpha_k^0+\tilde\alpha_k$ is 
unique (there is no global $1$-form on $X=\pc$).

{\bf Step 3: finding $\delta_k$.}
Since the trace connection must be $\zeta$, we have to set $\delta_k:=\omega_k-\alpha_k$.
It is straighforward that it satisfies the 3rd equation of \ref{CompConExpand} and the correct
residual term of \ref{ResidueCondition}. Actually, the sum of 2nd and 3rd equations of \ref{CompConExpand}
exactly give the compatibility condition of $(\omega_k=\alpha_k+\delta_k)_k$ forming $\zeta$. 

{\bf Step 4: finding $\beta_k$ :}
The right-hand-side of 4th equation of \ref{CompConExpand} defines an element of 
$H^1(X,L^{-1}\otimes\Omega_X^1)=\{0\}$. We can solve $\lambda a_ldb_{kl}+(b_{kl}a_l)(\alpha_k-\delta_l)=\tilde\beta_k-\tilde\beta_l$ with $\tilde\beta_k$ belonging to $L^{-1}\otimes\Omega_X^1$, so that $\beta_k:=\frac{\tilde\beta_k}{a_k}$
are sections of $\Omega_X^1$, thus satisfying the residual condition \ref{ResidueCondition}.

We have constructed a unique $\lambda$-connection from data $\gamma$ and $(b_{kl}a_l)_{kl}$.

{\bf Locus of Higgs fields.}
By Serre Duality, we have a perfect pairing
$$
H^0(X,L\otimes\Omega_X^1(D))\times H^1(X,L^{-1}(-D))\longrightarrow H^1(X,\Omega_X^1)\stackrel{\sim}{\longrightarrow}\mathbb C.
$$
More precisely, in our construction, to the data $(\ \gamma\ ,\ (b_{kl}a_l)_{kl}\ )$, we associate
the cocycle $(b_{kl}a_l)\gamma\in H^1(X,\Omega_X^1)$ which admits the meromorphic
resolution $(\alpha_k-\alpha_l)_{kl}$. The principal (polar) part of $(\alpha_k)_k$ is well-defined;
for instance, $\mathrm{Res}_{t_i}(\alpha_k)_k=\nu_i^-$ does not depend on the chart $U_k$.
The last arrow is given by the sum of residues: it measures the obstruction to realize the principal 
part by a global meromorphic $1$-form. Concretely, the image is 
$$\sum_{i=1,\ldots,n}\mathrm{Res}_{t_i}(\alpha_k)_k=\lambda\cdot\sum_{i=1,\ldots,n}\nu_i^-.$$
We get a Higgs field precisely when $\lambda=0$, i.e. when the image is zero. Finally, the locus of Higgs fields
in our theorem is given by the incidence variety for the above Serre Duality.
\end{proof}

\begin{figure}[htbp]
\begin{center}
\input{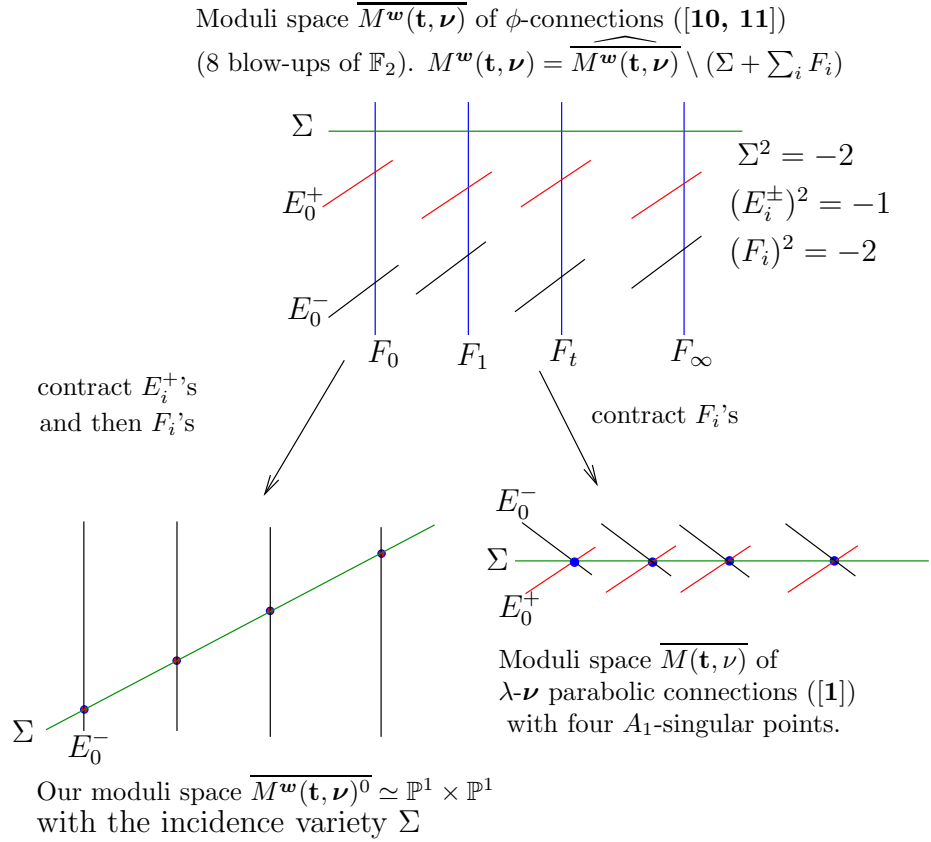} 
\caption{Different moduli spaces  and their relations  in case $n=4$.}
\label{figure:Moduli4}
\end{center}
\end{figure}

\begin{remark}\label{rem:smoothness}
{\rm 
 Note that without the condition of $(E, \bl) \in P^{\bw}_{-1}$ for our choice of the weights $\bw$ 
or the conditions in Proposition \ref{Prop:OurMainChart}, the coarse moduli space $\overline{M(\bt, \bnu)}$  
of $\lambda$-$\bnu$-connections  have singularities.  
We will explain about this in the case of $n=4$. (See Figure \ref{figure:Moduli4}).   
In this case, the coarse moduli space $ \widehat{\overline{M^{\bw}(\bt, \bnu)}}$ of stable parabolic $\bnu$-$\phi$-connections  gives a smooth compactification of the moduli space $M^{\bw}(\bt, \bnu)$ of $\bw$-stable $\bnu$-parabolic connections and it gives an Okamoto-Painlev\'{e} pair for Painlev\'e VI equations (\cite{SaitoTakebeTerajima,InabaIwasakiSaito,InabaIwasakiSaito2}). In fact, $M^{\bw}(\bt, \bnu)$ is the complement of 
$\Sigma+ \sum_{i} F_i$ in $ \widehat{\overline{M^{\bw}(\bt, \bnu)}}$. 
This moduli space  $ \widehat{\overline{M^{\bw}(\bt, \bnu)}}$ is 
isomorphic to the blown-up of 8-points of $\mathbb F_2$  
as in Figure \ref{figure:Moduli4}.  
 Note that in this case, $\phi$ is the endomorphism of $E =\cO_{X} \oplus \cO_{X}(-1)$. 
For simplicity, we assume that $\bnu$ is generic and all connections are $\bw$-stable.  
 Here the exceptional curves $E_i^{+} \setminus F_i \cap E_i^{+} $ (resp. $E_i^{-} \setminus  F_i \cap E_i^{-}$)  is  the locus of the parabolic connections such that  the apparent coordinate 
$q = t_i$ and $ l_i \subset \cO_{X}$ (resp. $l_i \not\subset \cO_{X}$) 
and $F_i \setminus (F_i \cap \Sigma)$ is the locus of $\phi$-connections with $\rank \phi =1 $. 
Moreover  $\Sigma$ is the locus of $\phi=0$, that is, the locus of 
Higgs bundles.  In order to obtain the moduli space $\overline{M(\bt, \bw)}$
of $\lambda$-connections (\cite{Arinkin}), we 
just contract $F_i$'s which are $(-2)$-rational curves.  Hence $\overline{M(\bt, \bw)}$ 
has four $A_1$-singular points.  On the other hand, our moduli space $\overline{M^{\bw}(\bt, \bnu)^0}$ can be obtained by contracting $E_{i}^{+}$'s and $F_i$'s and 
$\overline{M^{\bw}(\bt, \bnu)^0}$ is smooth and isomorphic to $\pp^1 \times \pp^1$ with the diagonal 
incidence variety $\Sigma \subset \pp^1 \times \pp^1$ as in Figure \ref{figure:Moduli4}.  
 }
\end{remark}

Another interpretation of our main theorem is that the image of the apparent map for Higgs fields
characterizes the bundle. 
Precisely, given $(E,\bl)\in P^{\bw}_{-1}(\bt)$, let us consider the fiber 
$\mathrm{Bun_H}^{-1}((E,\bl)) \subset $
of $\mathrm{Bun_H}$ in (\ref{eq:bd-bun}) and one can look at the restriction 
\begin{equation}\label{eq:boundary}
\mathrm{App}:\mathrm{Bun_H}^{-1}((E,\bl)) \simeq
\mathbb PH^0(\mathrm{End}(E,\boldsymbol{l})\otimes \Omega_{X}(D))\lra \vert\mathcal O_{X}(n-3)\vert
\end{equation}
to the boundary at infinity of connections $M^{\bw}(\bt,\bnu)^0$. Our main results says first that 
the image of ( \ref{eq:boundary} ) is non degenerate, i.e.  defines an hyperplane in $\vert\mathcal O_{X}(n-3)\vert$, 
thus defining an element of the dual $\vert\mathcal O_{X}(n-3)\vert^*$;
moreover, this hyperplane determines the parabolic structure $\bl$. 

\begin{corollary}\label{Coro:CanonicalParabolicMap}
The map 
$$
P^{\bw}_{-1}(\bt)\lra\vert\mathcal O_{X}(n-3)\vert^*\ ;\ (E,\bl)\mapsto \mathrm{image}(\mathrm{App}(\mathbb PH^0(\mathrm{End}(E,\boldsymbol{l})\otimes \Omega_{X}(D))))
$$
is well-defined and is an isomorphism.
\end{corollary}

In fact, it is not difficult to deduce our main result from this corollary; for instance,
that the above map is well-defined shows the injectivity of $\mathrm{App}\times\mathrm{Bun}$
in restriction to each fiber $\mathrm{Bun}^{-1}(E,\bl)$.
We will provide an alternate proof of this Corollary by direct computation in the next section.

\subsection{The degenerate case}When $\sum_i \nu_i^-=0$ ($\Leftrightarrow\sum_i \nu_i^+=1$), we get

\begin{proposition}\label{prop:ReducibleFibrations}
If $\sum_i \nu_i^-=0$, then $M^{\bw}(\bt,\bnu)^0$ identifies with the total space of the cotangent bundle $T^*V$,
and the map $\mathrm{Bun}:M^{\bw}(\bt,\bnu)^0\to V$, with the natural projection $T^*V\to V$.
Here, the section $\nabla_0:V\to M^{\bw}(\bt,\bnu)^0$ corresponding to the zero section of $T^*V\to V$ 
is given by those reducible connections
preserving the destabilizing subbundle $\mathcal O_X$.
Moreover, the map $\mathrm{App}\times\mathrm{Bun}$ is the natural map 
between total spaces 
$$
\xymatrix{ M^{\bw}(\bt,\bnu)^0\simeq T^*V \ar@{-->}[rr]^{\mathrm{App}\times\mathrm{Bun}} \ar[rd]_{\mathrm{Bun}} && 
\mathbb P T^*V\simeq\Sigma \ar[ld]^{\mathrm{Bun}} \\ & V }
$$
with indeterminacy locus $\nabla_0$.
\end{proposition}

Here, the restriction $(\nabla_0)_{\vert\mathcal O_X}$ has eigenvalues $\nu_i^-$
and Fuchs relation is just $\sum_i \nu_i^-=0$. Fibers of $\mathrm{App}\times\mathrm{Bun}$
are one-dimensional in this case.

\begin{proof}Going back to the proof of Theorem \ref{th:Main}, we see that, under assumption 
$\gamma=0$ (reducibility condition) and setting $\lambda=1$ we get a unique connection $\nabla_0$
for each given parabolic bundle $(E,\bl)\in P^{\bw}_{-1}(\bt)$. This section $\nabla_0$ allows to reduce the group 
structure of the affine bundle $\mathrm{Bun}:M^{\bw}(\bt,\bnu)^0\to V$: in this case,
$M(\bt,\bnu)^0$ is the trivial affine extension of the cotangent bundle $T^*V$,
i.e. the total space of the cotangent bundle itself, and  $\mathrm{Bun}$ is the 
natural projection $T^*V\to V$. 

The unique $\lambda$-connection with vanishing apparent map $\varphi_\nabla$ 
(or $\gamma$ if we follow again the proof of Theorem \ref{th:Main}) is 
$\lambda\cdot\nabla_0$. The apparent map of the general $\lambda$-connection
$\nabla=\lambda\cdot\nabla_0+\Theta$ is thus given by $\varphi_{\nabla}\equiv\varphi_{\Theta}$
by linearity. On each fiber $\mathrm{Bun}^{-1}(E,\bl)$, the apparent map is thus the projection
on the image of Higgs bundles, namely $\Sigma$.
\end{proof}

\section{Some computations}

Here we provide some explicit formulae for the two fibrations.

\subsection{Higgs fields and connections}\label{subsection:HiggsFields}

By fractional linear transformation, set $(t_{n-2},t_{n-1},t_n)=(0,1,\infty)$ for simplicity.
In order to describe the generic Higgs bundle or connection in matrix form, we use the following isomorphism
$$\mathrm{Elm}_{t_n}^-:P_0(\bt)\to  P_{-1}(\bt);(E,\bl)\mapsto(E',\bl').$$
It induces a birational map $U\dashrightarrow V$ between the projective charts $U:=U_{n-2,n-1,n}$ introduced 
in Section \ref{subsection:trivialbundle} and $V:=P_{-1}^{\bw}(\bt)$, in Section \ref{subsection:chartV}.
Precisely, denote by $e_1$ and $e_2$ a basis of $E=\mathcal O_X\oplus\mathcal O_X$
with $l_{n}=\cc\cdot e_1$, $l_{n-2}=\cc\cdot e_2$ and $l_{n-1}=\cc\cdot(e_1+e_2)$; 
then, choose the basis $(e_1',e_2')$ for $(E',\bl'):=\mathrm{Elm}_{t_n}^-(E,\bl)$ 
given by $e_i':=e_i$ outside of $t_n$: $e_2'$ has a pole at $t_n$ and generates $\mathcal O_X(-1)$.
Note that $e_1=e_1'$ is the cyclic vector for the apparent map.
The rational map $U\dashrightarrow V$ is therefore given by 
$$(u_1,\ldots,u_{n-3},0,1,\infty)\mapsto (v_1,\ldots,v_n)=(u_1,\ldots,u_{n-3},0,1,0)$$
where parabolic structures $\bl$ and $\bl'$ are respectively generated by
$u_ie_1+e_2$ and $v_ie_1'+e_2'$. We interpret this as a map
$$\begin{matrix}
U\simeq(\pc)^{n-3}&\dashrightarrow& V\simeq \pp^{n-3}_{\cc}\\
\mathbf{u}=(u_1,\ldots,u_{n-3})&\mapsto&\mathbf{v}=(u_1:\ldots:u_{n-3}:1)
\end{matrix}$$
Assume, for computations, that $(u_1,\ldots,u_{n-3})\in\cc^{n-3}$.
We also rename spectral data as follows:
$$
\begin{matrix}
E=\mathcal O_X\oplus\mathcal O_X & E'=\mathcal O_X\oplus\mathcal O_X(-1)\\
\begin{pmatrix}t_1&\cdots&t_{n-3}&0&1&\infty\\
\nu_{t_1}^+&\cdots&\nu_{t_{n-3}}^+&\nu_0^+&\nu_1^+&\nu_\infty^-\\
\nu_{t_1}^-&\cdots&\nu_{t_{n-3}}^-&\nu_0^-&\nu_1^-&\nu_\infty^+-1
\end{pmatrix}
&
\begin{pmatrix}t_1&\cdots&t_{n-3}&0&1&\infty\\
\nu_{t_1}^+&\cdots&\nu_{t_{n-3}}^+&\nu_0^+&\nu_1^+&\nu_\infty^+\\
\nu_{t_1}^-&\cdots&\nu_{t_{n-3}}^-&\nu_0^-&\nu_1^-&\nu_\infty^-
\end{pmatrix}
\end{matrix}
$$ 
Then, the general connection on $(E,\bl)$ or $(E',\bl')$ writes 
$$\nabla=\nabla_0+c_1\Theta_1+\cdots+c_{n-3}\Theta_{n-3},\ \ \ (c_i)\in\cc^{n-3}$$
where 
\begin{equation}\label{TheConnection}
\nabla_0:=d+\begin{pmatrix}\nu_0^-&0\\\rho&\nu_0^+\end{pmatrix}\frac{dz}{z}
+\begin{pmatrix}\nu_1^--\rho&\nu_1^+-\nu_1^-+\rho\\ -\rho&\nu_1^++\rho\end{pmatrix}\frac{dz}{z-1}
\end{equation}
$$
+\sum_{i=1}^{n-3}\begin{pmatrix}\nu_{t_i}^-&(\nu_{t_i}^+-\nu_{t_i}^-)u_i\\ 0&\nu_{t_i}^+\end{pmatrix}\frac{dz}{z-t_i},\ \ \ 
\text{with}\ \rho= \nu_{0}^-+ \nu_{1}^-+ \nu_{\infty}^-+
\sum_{i=1}^{n-3}\nu_{t_i}^-,
$$
and 
\begin{equation}\label{generatorHiggs}
\Theta_i:=\begin{pmatrix}0&0\\1-u_i&0\end{pmatrix}\frac{dz}{z}
+\begin{pmatrix}u_i&-u_i\\ u_i&-u_i\end{pmatrix}\frac{dz}{z-1}
+\begin{pmatrix}-u_i&u_i^2\\-1&u_i\end{pmatrix}\frac{dz}{z-t_i}.
\end{equation}
The connection $\nabla_0$ is the unique connection (compatible with the given parabolic structure) such that 
the divisor of the apparent map 
$\mathrm{App}(\nabla_0)$ takes the form $\mathrm{div}(\varphi_{\nabla_0})=t_1+\cdots+t_{n-3}$: 
in this case, $e_1$ is the $\nu_{t_i}^-$-eigendirection 
for $i=1,\ldots,n-3$. We note that 
$$\rho=0\Leftrightarrow
\nu_{0}^-+ \nu_{1}^-+ \nu_{\infty}^-+
\sum_{i=1}^{n-3}\nu_{t_i}^- = 0 
\Leftrightarrow \nu_{0}^++ \nu_{1}^++ \nu_{\infty}^++
\sum_{i=1}^{n-3}\nu_{t_i}^+=1$$
in which case $\nabla_0$ is the reducible connection (see Proposition \ref{prop:ReducibleFibrations}): 
the subbundle $\mathcal O_X\subset E$ (resp. $E'$)
generated by the cyclic vector $e_1$ (resp. $e_1'$) is $\nabla_0$-invariant.

The parabolic Higgs fields $\Theta_i$, $i=1,\ldots,n-3$, are independent over $\mathbb C$ 
(they do not share the same poles) and 
any other one is a linear combination of these $\Theta_i$'s.
These generators have been choosen so that the apparent map has divisor 
$\mathrm{div}(\varphi_{\Theta_i})=\mu_i+\sum_{j\not=i}t_j$ where $\mu_i = \frac{t_i(u_i-1)}{u_i- ti}$.
Moreover, the moduli space of parabolic Higgs bundles is naturally isomorphic to the total space of the cotangent bundle $T^*U$
(for those parabolic  Higgs bundles $(E,\bl,\Theta)$ with $(E,\bl)\in U$). 
Under this identification, $\Theta_i$
corresponds to the differential form $du_i$. 

Denoting $\nabla=d+Adz$, the apparent map is given by the coefficient $\varphi_\nabla:=A(2,1)$ and we get
\begin{equation}\label{varphiFormula}
\varphi_\nabla=-\frac{\rho}{z(z-1)}+\sum_{i=1}^{n-3}c_i\frac{(u_i-t_i)z+(1-u_i)t_i}{z(z-1)(z-t_i)} =
\frac{\tilde{\varphi}_{\nabla}(z)}{z(z-1)\prod_{j}(z-t_j)},
\end{equation}
where $\tilde{\varphi}_{\nabla}(z)$ is a polynomial of $z$ degree $n-3$.  
The roots $z=q_1,\ldots,q_{n-3}$
of $\varphi_\nabla$ (or of $\tilde{\varphi}_{\nabla}(z)$ ) are the apparent singular points with respect to the cyclic vector $e_1$ (resp. $e_1'$).
For such a variable $q$, we define the dual variable as
$$p:=A(1,1)\vert_{z=q}-\frac{\nu_0^-}{q}-\frac{\nu_1^-}{q-1}-\sum_{i=1}^{n-3}\frac{\nu_{t_i}^-}{q-t_i}$$
i.e.
\begin{equation}\label{pFormula}
p=-\frac{\rho}{q-1} +\sum_{i=1}^{n-3}c_iu_i\left(\frac{1}{q-1}-\frac{1}{q-t_i}\right).
\end{equation}
The natural symplectic structure on the moduli space $M(\bt,\bnu)$ is defined by 
\begin{equation}\label{eq:omega}
\omega=\sum_{i=1}^{n-3} dp_i\wedge dq_i
\end{equation}
and the two maps $\mathrm{App}$ and $\mathrm{Bun}$ are Lagrangian with respect to $\omega$.
Here recall that $(q_i, p_i)$ are not the coordinates for the moduli space $M(\bt, \bnu)$, 
but the coordinates for some $(n-3)!$-covering of $M(\bt, \bnu)$. However the symplectic form 
$\omega$ in (\ref{eq:omega}) is invariant under the changing the order of roots $q_i$, 
so it descends to a symplectic form on  $M(\bt, \bnu)$.

Under these explicit notation, we can give the following
\begin{proof}[Alternate proof of Corollary \ref{Coro:CanonicalParabolicMap}]
We want first to show
that the map $\varphi_\Theta$ is not identically zero for any Higgs bundle $(E,\bl,\Theta)$ with $(E,\bl)\in V$. 
A Higgs field writes in a matrix form as
$$\Theta=\begin{pmatrix}\alpha&\beta\\ \gamma&\delta\end{pmatrix}$$
and the map $\varphi_{\Theta}$ is given by the coefficient $\gamma$ which is a holomorphic section of $\mathcal O_{X}(n-3)$.
We want to check that $\gamma\equiv 0$ implies that either $\Theta\equiv0$, or one of the parabolics $l_i\in\mathcal O_{X}$.
If $\gamma\equiv 0$, then $\alpha$ and $\delta$ have to vanish over each $z=t_i$, $i=1,\ldots,n$, since 
$\Theta$ has to be nilpotent over these points. But $\alpha$ and $\delta$ are sections of $\mathcal O_{X}(n-2)$
and have thus to be identically zero. Finally, if $\beta\not\equiv0$, then, as a section of $\mathcal O_{X}(n-1)$, it cannot
vanish at all $z=t_i$: for some $t_i$ the matrix is not zero and take the form 
$$\Theta\vert_{z=t_i}=\begin{pmatrix}0&1\\0&0\end{pmatrix}$$
and the corresponding parabolic $l_i$ lies on $\mathcal O_{X}$. 
The map $\Theta\mapsto\varphi_{\Theta}$ defines an homomorphism
$$H^0(\mathrm{End}(E,\boldsymbol{l})\otimes \Omega_{X}(D))\to H^0(X,L\otimes\Omega_X^1(D))$$
and we have just proved that it has zero kernel: it is injective.
Therefore, its image $\mathrm{image}(\mathrm{App}(\mathbb PH^0(\mathrm{End}(E,\boldsymbol{l})\otimes \Omega_{X}(D))))$
defines an hyperplane of $H^0(X,L\otimes\Omega_X^1(D))$, i.e. an element of the dual $\mathcal O_{X}(n-3)\vert^*$,
which depends only on $(E,\boldsymbol{l})\in V$.


We have thus proved that the map $V\to\vert\mathcal O_{X}(n-3)\vert^*$ is a well-defined morphism
and may be viewed as an endomorphism of $\mathbb P_{\cc}^{n-3}$
(after fixing isomorphisms with $\mathbb P_{\cc}^{n-3}$). In order to prove that
it is an isomorphism, we just  have 
to check that it  is birational. For this, it is enough to prove that the composition 
$$U\ \stackrel{\mathrm{Elm}_{t_n}^-}{\dashrightarrow}\ V\ \stackrel{\mathrm{Bun}}{\longrightarrow}\ \vert\mathcal O_{X}(n-3)\vert^*$$
is birational (since the left-hand-side is). We compute this latter one in the affine chart
$(u_1,\ldots,u_{n-3})\in\mathbb C^{n-3}\subset U$. 
For $\Theta=\Theta_i$, the map $\varphi_{\Theta_i}$ is the multiplication by
$$\gamma=\frac{P_i(z)\cdot dz}{z(z-1)\prod_j(z-t_j)}\ \ \ \text{where}\ \ \ P_i(z)=\left[ (u_i-t_i)z+(1-u_i)t_i\right]\prod_{j\not=i}(z-t_j).$$
The zero  of $\varphi_{\Theta_i}$ is thus given by $P_i(z)=0$. Now, consider the line $\Delta_i$ defined in
$\vert\mathcal O_{X}(n-3)\vert$ by those polynomials vanishing on all $t_j$, $j\not=i$; these lines all intersect
at the single point defined by the very special polynomial $\prod_{i}(z-t_i)$ and span $\mathbb P_{\cc}^{n-3}$ ($t_i\not=t_j$ for any $i\not=j$).
For generic $u_i$'s, the hyperplane image $H\subset\mathbb P_{\cc}^{n-3}$ of $\mathrm{App}$ cuts out all $\Delta_i$'s
outside of their common intersection point. Conversely, a generic hyperplane $H$ cuts out each $\Delta_i$
at a single point defined by say $(z-\mu_i)\prod_{j\not=i}(z-t_j)$; solving $\mu_i=\frac{t_i(u_i-1)}{u_i-t_i}$
gives the parabolic structure $(u_1,\ldots,u_{n-3})$.
\end{proof}

Let us start with $\vert\mathcal O_{X}(n-3)\vert\simeq\pp_{\cc}^{n-3}$ equipped with the following projective
coordinates : $\mathbf{a}=(a_0:a_1:\cdots:a_{n-3})$ stands for the polynomial equation $a_{n-3}z^{n-3}+\cdots+a_1z+a_0=0$. 
It can be interesting
to view also this space as $\mathrm{Sym}^{n-3}X$ with $X=\pc$ our initial base curve, and we have a natural map
$$\mathrm{Sym}:X^{n-3}\to\mathrm{Sym}^{n-3}X\ ;\ (q_1,\ldots,q_{n-3})\mapsto (z-q_1)\cdots(z-q_{n-3}).$$
The dual $\vert\mathcal O_{X}(n-3)\vert^*$ is the set of hyperplanes 
$a_0b_0+a_1b_1+\cdots+a_{n-3}b_{n-3}=0$ and has thus natural projective coordinates
$\mathbf{b}=(b_0:b_1:\cdots:b_{n-3})$. Let us explicitely compute the relation ship between 
usual Darboux coordinates $(p_i,q_i)$, our basic coordinates $(u_i,c_i)$ and the new coordinates
$(\mathbf{a},\mathbf{b})$ from our main Theorem \ref{th:Main}. We do this for the Painlev\'e case $n=4$
and the first Garnier case $n=5$.

\subsection{Case $n=4$}

Our starting variables are $u\in\cc\subset U$ and $c\in\cc$. From (\ref{varphiFormula}) and (\ref{pFormula}),
we get Darboux variables:
$$p=-\frac{(t-u)(\rho+c(t-u))}{t(t-1)}\ \ \ \text{and}\ \ \ q=t\frac{\rho+c(1-u)}{\rho+c(t-u)};$$
reversing, we get:
$$u=t\frac{\rho+p(q-1)}{\rho+p(q-t)}\ \ \ \text{and}\ \ \ c=-\frac{(q-t)(\rho+c(q-t))}{t(t-1)}.$$
The apparent map for Higgs bundle (set $c=\infty$ in above formula) vanishes at 
$$\mu=t\frac{1-u}{t-u}=q+\frac{\rho}{p}$$ and we get
$$(a_1:a_0)=(1:-q)\ \ \ \text{and}\ \ \ (b_1:b_0)=(\mu:1).$$
The symplectic structure is given by 
$$dp\wedge dq=dc\wedge du=\rho\cdot d\left(\frac{a_0db_0+a_1db_1}{a_0b_0+a_1b_1}\right).$$
Our $\mu$-variable is exactly the $Q$-variable involved in Section 8 of \cite{LoraySaitoSimpson}
and it was observed, there, that Okamoto symmetry is just given by the involution $(q,\mu)\mapsto(\mu,q)$
(i.e. $(a_1:a_0)\leftrightarrow (b_0:-b_1)$) permuting the two fibrations. We will see in the next section
that there is no such symmetric (global on $M(\bt,\bnu)$) permuting the two fibrations for $n=5$.

\section{Computations for the case $n=5$}\label{sec:Parabolic5}

A straightforward computation shows that the map
$$\mathrm{Bun}\circ\mathrm{Elm}_{t_n}^-:\left\{\begin{matrix}
U&\to&V=\vert\mathcal O_{X}(n-3)\vert^*\\
(u_1,u_2)&\mapsto& (b_0:b_1:b_2)
\end{matrix}\right.$$
is given by
$$\left\{\begin{matrix}
b_2=&t_1t_2(t_1(t_2-1)u_1-(t_1-1)t_2u_2+(t_1-t_2))\\
b_1=&t_1t_2((t_2-1)u_1-(t_1-1)u_2+(t_1-t_2))\\
b_0=&t_2(t_2-1)u_1-t_1(t_1-1)u_2+t_1t_2(t_1-t_2)
\end{matrix}\right.$$
The $(u_1,u_2)$ affine chart may thus be seen as an affine chart of $V$, or equivalently, 
$V$ as an alternate compactification
of the $(u_1,u_2)$-chart. The inverse map is given by 
$$\left\{\begin{matrix}
u_1=&t_1\frac{b_2-(t_2+1)b_1+t_2b_0}{b_2-(t_1+t_2)b_1+t_1t_2b_0}\\
u_2=&t_2\frac{b_2-(t_1+1)b_1+t_1b_0}{b_2-(t_1+t_2)b_1+t_1t_2b_0}
\end{matrix}\right.$$
Apparent singular points are the roots of the polynomial 
$$P(z)=-\rho(z-t_1)(z-t_2) $$
$$+ c_1\left[(u_1-t_1)z+(1-u_1)t_1\right](z-t_2)+c_2(z-t_1)\left[(u_2-t_2)z+(1-u_2)t_2\right]$$

$$=\left[c_1(u_1-t_1)+c_2(u_2-t_2)-\rho\right]z^2$$
$$+\left[\rho(t_1+t_2)+c_1(t_1(t_2+1)-u_1(t_1+t_2)+c_2((t_1+1)t_2-u_2(t_1+t_2)\right]z$$
$$+t_1t_2\left[c_1(u_1-1)+c_2(u_2-1)-\rho\right]$$
We can re-write
$$P(z)=\rho[b_2-(t_1+t_2)b_1+t_1t_2b_0](z-t_1)(z-t_2)$$
$$+c_1t_1(t_1-1)(b_2-(z+t_2)b_1+zt_2b_0)(z-t_1) $$
$$+c_2t_2(t_2-1)(b_2-(z+t_1)b_1+zt_1b_0)(z-t_2).$$
Denoting $z=q_1$ and $z=q_2$ the two apparent singular points, we get
$$c_1=\rho\ \frac{(q_1-t_1)(q_2-t_1)}{t_1(t_1-1)(t_1-t_2)}\ \frac{b_2-(t_1+t_2)b_1+t_1t_2b_0}{b_2-(q_1+q_2)b_1+q_1q_2b_0}$$
and
$$c_2=\rho\ \frac{(q_1-t_2)(q_2-t_2)}{t_2(t_2-1)(t_2-t_1)}\ \frac{b_2-(t_1+t_2)b_1+t_1t_2b_0}{b_2-(q_1+q_2)b_1+q_1q_2b_0}$$
and we already see strong transversality between parabolic and apparent fibrations.
In a more symmetric way, we can introduce the equation $a_2q^2+a_1q+a_0=0$ of the two apparent singular points 
and we get the following formula
 $$c_1=\rho\ \frac{a_2t_1^2+a_1t_1+a_0}{t_1(t_1-1)(t_1-t_2)}\ \frac{b_2-(t_1+t_2)b_1+t_1t_2b_0}{a_2b_2+a_1b_1+a_0b_0}$$
and
$$c_2=\rho\ \frac{a_2t_2^2+a_1t_2+a_0}{t_2(t_2-1)(t_2-t_1)}\ \frac{b_2-(t_1+t_2)b_1+t_1t_2b_0}{a_2b_2+a_1b_1+a_0b_0}.$$
As expected by our choice of coordinates, the locus of Higgs bundles, where $c_1$ and/or $c_2$ goes to the infinity,
is given by the incidence variety $a_2b_2+a_1b_1+a_0b_0=0$.
For each root $z=q_i$, the dual variable is expressed by
$$p_i=-\frac{\rho}{q_i-1} +c_1u_1\left(\frac{1}{q_i-1}-\frac{1}{q_i-t_1}\right)+c_2u_2\left(\frac{1}{q_i-1}-\frac{1}{q_i-t_2}\right)$$
and we get
$$p_1=\rho\frac{b_1-q_2b_0}{b_2-(q_1+q_2)b_1+q_1q_2b_0}\ \ \ \text{and}\ \ \ 
p_2=\rho\frac{b_1-q_1b_0}{b_2-(q_1+q_2)b_1+q_1q_2b_0}.$$
We find that
$$\eta:=p_1 dq_1+p_2 dq_2=\rho\frac{a_2 db_2+a_1 db_1+a_0 db_0}{a_2b_2+a_1b_1+a_0b_0}
-\rho\frac{d(a_2b_2+a_1b_1+a_0b_0)}{a_2b_2+a_1b_1+a_0b_0}+\rho\frac{da_2}{a_2}$$
where $(a_2:a_1:a_0)\sim(1:-q_1-q_2:q_1q_2)$.
The differential
$$\omega:=d\eta=dp_1 \wedge dq_1+dp_2\wedge dq_2
=\rho\cdot d\left(\frac{a_2 db_2+a_1 db_1+a_0 db_0}{a_2b_2+a_1b_1+a_0b_0}\right)$$
is anti-invariant under the involution $(a_2:a_1:a_0)\leftrightarrow(b_2:b_1:b_0)$ that exchanges the two sets of projective coordinates.

A straightforward computation shows that, pulling-back the symplectic form $\omega$ to our initial parameters
$c_i$ and $u_i$, we obtain
$$dp_1\wedge dq_1+dp_2\wedge dq_2=dc_1\wedge du_1+dc_2\wedge du_2$$
which is the Liouville form on moduli space of Higgs bundles.

We have also the following formula comparing classical coordinates to our parabolic ones:
$$(b_2:b_1:b_0)=(p_1q_1^2-p_2q_2^2+\rho(q_1-q_2)\ :\ p_1q_1-p_2q_2\ :\ p_1-p_2).$$

\subsection{The coarse moduli space of parabolic bundles and Del Pezzo geometry}

Here, we want to explicitely describe the full coarse moduli space $P_{-1}(\bt)$ of all 
undecomposable parabolic bundles
as a finite union of projective charts patched together by birational maps between open subsets
(see end of Section \ref{sec:GITbundle}). We already get our main projective chart $V\subset P_{-1}(\bt)$
(defined in Section \ref{subsection:chartV})
that contains almost all undecomposable parabolic bundles: $V:=P_{-1}^{\bw}(\bt)\simeq\pp_\bb^{2}$ where $\bw=(w,w,w,w,w)$ 
with $\frac{1}{5}<w<\frac{1}{3}$.
To get the full coarse moduli space $P_{-1}(\bt)$, 
we have to add those undecomposable parabolic structures on $\mathcal O_{X}\oplus\mathcal O_{X}(-1)$ 
with $1$ or $2$ parabolics lying on $\mathcal O_{X}$, and the unique undecomposable parabolic structure 
on $\mathcal O_{X}(1)\oplus\mathcal O_{X}(-2)$. They will occur in $P_{-1}(\bt)$ as points infinitesimally close 
to special points of $V=\pp_\bb^{2}$, namely those non generic bundles (see Section \ref{sec:GITbundle}).
Let us list them.

\begin{table}[htdp]
\begin{center}
\begin{tabular}{|c|c|c|}
\hline 
& $V$ & $P_{-1}(\bt)\setminus V$ \\
\hline
5 & $D_i :\left\{\begin{matrix}E=\mathcal O_{X}\oplus\mathcal O_{X}(-1)\\
 l_j,l_k,l_m,l_n \subset \mathcal O_{X}(-1)\end{matrix}\right.$ & 
 $\Pi_i:\left\{\begin{matrix}E=\mathcal O_{X}\oplus\mathcal O_{X}(-1)\\ l_i \subset \mathcal O_{X}\end{matrix}\right.$\\
\hline
10 & $\Pi_{i,j}:\left\{\begin{matrix}E=\mathcal O_{X}\oplus\mathcal O_{X}(-1)\\ l_k,l_m,l_n\subset\mathcal O_{X}(-1)\end{matrix}\right.$ & 
$D_{i,j}:\left\{\begin{matrix}E=\mathcal O_{X}\oplus\mathcal O_{X}(-1)\\ l_i,l_j
\subset\mathcal O_{X}\end{matrix}\right.$\\
\hline
1 & $\Pi:\left\{\begin{matrix}E=\mathcal O_{X}\oplus\mathcal O_{X}(-1)\\ l_i,l_j,l_k,l_m,l_n\subset
\mathcal O_{X}(-2)\end{matrix}\right.$ & $D:\left\{\begin{matrix}E=\mathcal O_{X}(1)\oplus\mathcal O_{X}(-2)\\ - \end{matrix}\right.$\\
\hline
\end{tabular}
\end{center}
\caption{Non generic bundles (here, $\{i,j,k,m,n\}=\{1,2,3,4,5\}$).}
\label{default}
\end{table}

{\bf Defining $\Pi_i$ and $D_i$.}
The locus $\Pi_i$ of those undecomposable parabolic structures on $E=\mathcal O_{X}\oplus\mathcal O_{X}(-1)$ having exactly 
$l_i\subset\mathcal O_{X}$ (other parabolics outside $\mathcal O_{X}$) is naturally isomorphic to $X$: it is the moduli
space of undecomposable parabolic structures on $E$ over the $4$ other points, none of them lying on $\mathcal O_{X}$
(see Section \ref{subsec:n=4detailled}).
Each of these parabolic bundles is infinitesimally close to the unique undecomposable parabolic structures on $\mathcal O_{X}\oplus\mathcal O_{X}(-1)$ with all $l_j$, $j\not=i$, lying on the same 
$\mathcal O_{X}(-1)\hookrightarrow E$ (see Section \ref{section:WallCrossing}). There is only one undecomposable bundle
with this latter property and it defines a single point $D_i\in V$. As we shall see, $\Pi_i$ will occur, when we pass 
to another projective chart for $P_{-1}(\bt)$, as the exceptional
divisor after blowing-up the point $D_i$ (wall-crossing phenomenon).

{\bf Defining $\Pi_{i,j}$ and $D_{i,j}$.}
There is a unique undecomposable  parabolic structure $\boldsymbol{l}$ on $E=\mathcal O_{X}\oplus\mathcal O_{X}(-1)$ having exactly 
$l_i,l_j\in\mathcal O_{X}$ (other parabolics outside $\mathcal O_{X}$). It is infinitesimally close to the 
one-parameter family of undecomposable parabolic structures on $\mathcal O_{X}\oplus\mathcal O_{X}(-1)$ with all $l_k$, $k\not=i,j$, 
lying on the same $\mathcal O_{X}(-1)\hookrightarrow E$; this latter family form a rational curve $\Pi_{i,j}\subset V$
which is also naturally parametrized by $X$. Indeed, there is also a $\mathcal O_{X}(-1)\hookrightarrow E$ passing through
$l_i$ and $l_j$, and these two embeddings intersect over a point $z\in X$.
The locus of $\boldsymbol{l}$ is given by a single point $D_{i,j}\not\in V$ in the coarse moduli space $P_{-1}(\bt)$ which is infinitesimally close to any point of $\Pi_{i,j}$. When switching to some other projective charts of $P_{-1}(\bt)$, by moving weights,
the rational curve $D_{i,j}$ is eventually contracted, replaced by the single point $D_{i,j}$.

{\bf Defining $\Pi$ and $D$.}
Finally, the unique undecomposable parabolic structure on $\mathcal O_{X}(1)\oplus\mathcal O_{X}(-2)$ is infinitesimally close to
the one-parameter family of parabolic structures on $E=\mathcal O_{X}\oplus\mathcal O_{X}(-1)$ with all parabolics
lying on the same $\mathcal O_{X}(-2)\hookrightarrow E$. The latter family is again a rational curve $\Pi\subset V$
naturally parametrized by $X$: the subbundles $\mathcal O_{X}(-2)$ and $\mathcal O_{X}$ coincide over a unique point of $X$.
The undecomposable parabolic bundle $\mathcal O_{X}(1)\oplus\mathcal O_{X}(-2)$ is thus represented by a single point $D\in P_{-1}(\bt)\setminus V$
infinitesimally closed to $\Pi$.

{\bf Computations in $V=\pp_\bb^{2}$.}
We have summarized the list of non generic undecomposable parabolic bundles in the table above. All $\Pi,\Pi_i,\Pi_{i,j}$
are one-parameter families naturally parametrized by $X$; they form rational curves in $P_{-1}(\bt)$.
All $D,D_i,D_{i,j}$ are just points. On each line, bundles from the two columns are infinitesimally closed;
bundles from the left side are contained in the main chart $V$ while those on the right side are outside.
There are $16$ one parameter families of special bundles infinitesimally closed to $16$ special bundles.
A straightforward computation shows that:
\begin{itemize}
\item $\Pi$ is the conic with equation $b_1^2-b_0b_2$ parametrized by 
$$X\to \Pi\subset\pp_\bb^{2}\ ;\ z\mapsto (1:z:z^2).$$
\item $D_i$ is the image of $z=t_i$ through the previous mapping.
\item $\Pi_{i,j}$ is the line passing through $D_i$ and $D_j$.
\end{itemize}

\begin{figure}[htbp]
\begin{center}

\input{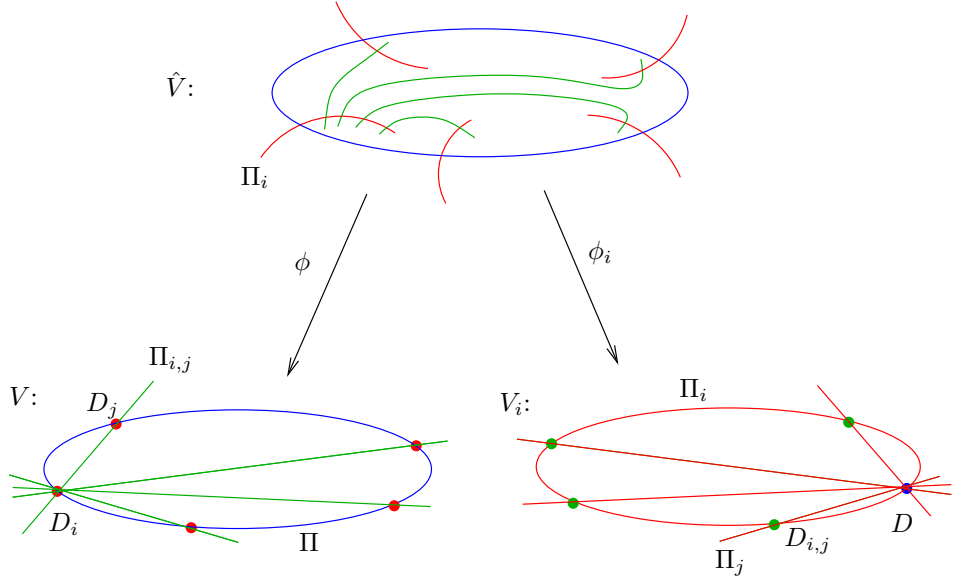}
 
\caption{Projective charts $V$, $\hat V$ and $V_i$'s and Del Pezzo geometry}
\label{figure:DelPezzo}
\end{center}
\end{figure}

{\bf Moving democratic weights}
Let us consider the moduli space $P_{-1}^{\bw}(\bt)$ for ``democratic'' weights $w_i=w$, $i=1,\ldots,n$,
and see how $P_{-1}^{\bw}(\bt)$ is varrying while $w$ goes from $0$ to $1$. At the beginning, when 
$w<\frac{1}{5}$, the weight is not admissible (see Section \ref{sec:GITbundle}) and $P_{-1}^{\bw}(\bt)=\emptyset$. 

For $\frac{1}{5}<w<\frac{1}{3}$,
$P_{-1}^{\bw}(\bt)=V$ is the main projective chart discussed above, the projective plane $\mathbb P_{\cc}^2$ with coordinates
$(b_0:b_1:b_2)$. 

When we pass to the next admissible chamber $\frac{1}{3}<w<\frac{3}{5}$, all five bundles $D_i$'s become unstable
while all five families $\Pi_i$'s now consist in stable bundles. The new moduli space of stable bundles $P_{-1}^{\bw}(\bt)$
is obtained from the previous one by blowing-up the five points $D_i$'s, that are replaced by the corresponding $\Pi_i$'s in the moduli space. We thus get a $5$-points blow-up of $\mathbb P_{\bb}^2$, a Del Pezzo surface of degree $4$.   Let us denote by $\hat V$ the Del Pezzo surface and $\phi:\hat V \to V \simeq \pp_{\bb}^2$ the blowing-up of 
five points  $D_i$'s.
  There are exactly $16$ rational curves having $-1$ self-intersection on it, namely the $5$ exceptional divisors 
$\Pi_i$'s (arising from blowing-up the $D_i$'s) and the strict transforms of 
the conic $\Pi$ and of the $10$ lines $\Pi_{i,j}$:
they precisely correspond to our $16$ families of special bundles (see Figure \ref{figure:DelPezzo}).

Finally, when we pass to the last chamber $\frac{3}{5}<w<1$, bundles of the family $\Pi$ become unstable
and the parabolic structure $D$ on the special bundle $E=\mathcal O_{X}(1)\oplus\mathcal O_{X}(-2)$ becomes stable.
The corresponding moduli space $P_{-1}^{\bw}(\bt)$ is thus obtained by contracting the $(-1)$-curve $\Pi\in\hat V$
onto the single point $D$.
This may be viewed as a $4$-points blow-up of $\mathbb P_{\cc}^2$ (the degree $5$ Del Pezzo surface).

{\bf Patching two charts}
Let us now focus on the two projective charts $V$ and $\hat V$. 
Restricting the blowing-up $\phi:\hat V \to V$ to the complement of $\Pi_i$, 
we have a natural isomorphism
$$
\phi^{0}:\hat V\setminus (\Pi_1\cup \Pi_2\cup \Pi_3\cup \Pi_4\cup \Pi_5)\stackrel{\sim}{\lra} V\setminus\{D_1,D_2,D_3,D_4,D_5\}
$$
that identify equivalence classes of bundles that both occur in $\hat V$ and $V$. Indeed, when $w$ crosses 
the special value $\frac{1}{3}$, sign of stability index changes only for those bundles $D_i$'s and $P_i$'s.
The map $\phi$ obviously contracts each $\Pi_i$ to $D_i$, it is the blow-up morphism. Now, the 
non reduced scheme obtained by patching together $V$ and $\hat V$ by $\phi^0$ is still an open subset 
of the  moduli space  $P_{-1}(\bt)$ that contains both $V$ and the non separated locus $\Pi_i$:
$$V\hookrightarrow V\cup_{\phi^{0}} \hat V\hookrightarrow P_{-1}(\bt).$$
In other words, by patching $\hat V$ to $V$ like above, we have added to $V$ the non separated 
points of $\Pi_i$'s. Mind that not only generic bundles are identified by $\phi^{0}$, also generic points 
of $\Pi$, for example, occur in both charts and are identified by $\phi^{0}$.

{\bf Geometry of Del Pezzo}
We want now to cover the full coarse moduli space $P_{-1}(\bt)$ by a finite number of smooth projective
charts and patch them together like we have just done. There are many projective charts
$V'\subset P_{-1}(\bt)$ that can be defined as moduli space of stable bundles $P_{-1}^{\bw}(\bt)$
(many chambers in the space of weights) but we do not need all of them to cover the coarse moduli space $P_{-1}(\bt)$.
We will use the classical geometry of the Del Pezzo surface $\hat V$ in order to select a few number of them.
First of all, note that $\hat V$ is dominating all other projective charts in the following sense:
if $V'$ is another chart, the natural birational map $\phi':\hat V\to V$ is actually a morphism.
Indeed, $V'$ only differ from $\hat V$ by the fact that some of the one-parameter families $\Pi'$
are contracted to points $D'$.

The Del Pezzo surface $\hat V$ contains $16$ rational $(-1)$-curves
that correspond, in our modular setting, to the $16$ families of special bundles.
Each of these ``lines'' intersects $5$ other ones with cross-ratio determined by the $t_i$'s. 
Recall that each of these curves are naturally parametrized by $X\supset\{t_1,t_2,t_3,t_4,t_5\}$
in our modular interpretation, and each poles $t_i$ corresponds to intersecting lines.

Appart from special symmetric values of the $t_i$'s,
the automorphism group of this Del Pezzo surface has order $16$ and it acts $1$-transitively on 
$-1$ rational curves: given $2$ of these lines $\Pi,\Pi'\subset\hat V$, there is exactly one automorphism
$\phi:\hat V\to\hat V$ sending $\Pi$ to $\Pi'$. This group has also modular interpretation. If we apply
two elementary transformations (see Section \ref{sec:twist&elm}), we get an automorphism
$$\mathrm{Elm}_{t_i}^-\circ\mathrm{Elm}_{t_j}^+: P_{-1}(\bt)\stackrel{\sim}{\lra}P_{-1}(\bt).$$
This automorphism must permute special points $D,D_i,D_{i,j}$'s and permute special families $\Pi,\Pi_i,\Pi_{i,j}$
and induces, in particular, an automorphism 
$$(\mathrm{Elm}_{t_i}^-\circ\mathrm{Elm}_{t_j}^+)\vert_{\hat V}: \hat V\stackrel{\sim}{\lra}\hat V$$
of the projective chart $\hat V$. It is straightforward to check 
that the group generated by these automorphisms has order $16$ and acts $1$-transitively on the $16$ lines.
Precisely, we have
$$(\mathrm{Elm}_{t_i}^-\circ\mathrm{Elm}_{t_j}^+): \Pi\stackrel{\sim}{\lra}\Pi_{i,j}$$
for all $i,j$ and 
$$(\mathrm{Elm}_{t_j}^-\circ\mathrm{Elm}_{t_k}^+)\circ(\mathrm{Elm}_{t_m}^-\circ\mathrm{Elm}_{t_n}^+): \Pi\stackrel{\sim}{\lra}\Pi_i$$
where $\{i,j,k,m,n\}=\{1,2,3,4,5\}$.

There are $16$ ways to go back to $\mathbb P_{\cc}^2$ by blowing-down $5$ curves in $\hat V$:
we have to choose $5$ lines intersecting a given one $\Pi'$. All these $\mathbb P_{\cc}^2$ correspond to 
moduli spaces for different choices of weights. Indeed, after an even number of elementary tranformations,
we can assume $\Pi'=\Pi$. So the chart $V'\simeq\mathbb P_{\cc}^2$ obtained by contracting those $5$ lines
intersecting $\Pi'$  in $\hat V$ is given as moduli space $V'=P_{-1}^{\bw}(\bt)$ with weights of the form
$w_i\in\{w,1-w\}$, $\frac{1}{5}<w<\frac{1}{3}$, with an even number of occurences $w_i=1-w$; note that 
there are $16$ such possibilities and we denote them by $V,V_i,V_{i,j}$ in the obvious way.

{\bf The whole moduli  space $P_{-1}(\bt)$}
Consider, like before, the projective charts $V_i\simeq\mathbb P_{\cc}^2$ obtained by contracting
in $\hat V$ the $5$ lines intersecting $\Pi_i$. The automorphism 
$$(\mathrm{Elm}_{t_j}^-\circ\mathrm{Elm}_{t_k}^+)\circ(\mathrm{Elm}_{t_m}^-\circ\mathrm{Elm}_{t_n}^+): P_{-1}(\bt)\stackrel{\sim}{\lra}P_{-1}(\bt)$$
where $\{i,j,k,m,n\}=\{1,2,3,4,5\}$, permutes the lines $\Pi$ and $\Pi_i$, and thus the charts $V$ and $V_i$;
it follows that $V_i=P_{-1}^{\bw}(\bt)$ for weights of the form
$w_i=w$ and $w_j=w_k=w_m=w_n=1-w$ with $\frac{1}{5}<w<\frac{1}{3}$.
It is then easy to check that these charts are enough to cover the whole coarse moduli space. Precisely, we have:
\begin{itemize}
\item $D\in V_i$ for all $i$,
\item $D_i\in V$,
\item $D_{i,j}\in V_i$ and $V_j$,
\item $\Pi\subset V,\hat V$,
\item $\Pi_i\subset \hat V$ and all $V_j$,
\item $\Pi_{i,j}\subset V$, $\hat V$ and all $V_k\not=V_i,V_j$.
\end{itemize}
We finally obtain the following description:
\begin{equation}\label{eq:patching}
P_{-1}(\bt)=\hat V \cup V\cup V_1\cup V_2\cup V_3\cup V_4\cup V_5
\end{equation}
where $\cup$ means that we identify all isomorphism classes of bundles that are shared 
by any two of these projective charts by means of the natural birational isomorphisms.
More precisely, we have the explicit blowing-down morphisms $\phi:\hat V\to V\simeq\mathbb P^2_\bb$
and $\phi_i:\hat V\to V_i\simeq\mathbb P^2_\cc$ introduced above and the patching (\ref{eq:patching})
are induced by $\phi$ and $\phi_i$ on the maximal open subsets where it is one-to-one.

The union of Proposition \ref{prop:patching5} is not sharp: we can delete one of the $V_i$'s, remaining
$6$ charts are enough to cover the whole moduli. However, we stress that we cannot delete $\hat V$.
Indeed, so far, we have not been very rigorous with those special $t_i$ points occuring
along our one-parameter families $X\stackrel{\sim}{\lra}\Pi,\Pi_i,\Pi_{i,j}$, namely those intersection
points between two such families. Let us explain on an example. The family $\Pi$ of those parabolic
structures $\bl$ on $E=\mathcal O_X\oplus\mathcal O(-1)$ such that all parabolics $l_i$'s are
contained in the same subbundle $\mathcal O_X(-2)\hookrightarrow E$. Such a subbundle is
determined, up to automorphism of $E$, by its intersection locus with the special subbundle $\mathcal O_X$:
there is a unique point $z\in X$ such that these two subbundles coincide over this point.
Since the parabolic structure is determined by $\mathcal O_X(-2)$, it is also determined 
by the corresponding point $z$ and we get a natural parametrization $X\stackrel{\sim}{\lra}\Pi$.
There are $5$ special points 
corresponding to the case where the two subbundles $\mathcal O_X(-2)$ and $\mathcal O_X$
coincide over $t_i$: then $l_i\subset\mathcal O_X$ and it is the intersection point with the family $\Pi_i$.
This special parabolic bundle in $\Pi$ occurs in $\hat V$, but not in our main chart $V$.
Indeed, stability assumption for $V$ excludes the possibility of $l_i\subset\mathcal O_X$,
and this special bundle is replaced by $D_i\in V$.

\subsection{The duality picture}

We now go back to the moduli space of connections $M^{\bw}(\bt,\bnu)$. An open subset $M^{\bw}(\bt,\bnu)^0$
is given by those connection $(E,\nabla,\boldsymbol{l})$ whose underlying parabolic bundle belongs to our main chart
$(E,\boldsymbol{l})\in V$. Recall that a natural compactification is obtained by adding projective Higgs bundles 
$$\mathrm{App}\times\mathrm{Bun}:\overline{M^{\bw}(\bt,\bnu)^0}\stackrel{\sim}{\lra}\pp_\ba^2\times\pp_\bb^2$$
and the boundary of $M^{\bw}(\bt,\bnu)^0$ corresponds to the incidence variety $\Sigma:\{a_0b_0+a_1b_1+a_2b_2=0\}$.
We would like to add to this picture those missing connections, i.e. the connections on missing parabolic bundles.

In order to do this, let us denote by $C$ the image of the diagonal through the map
$$\mathrm{Sym}:X\times X\to\mathrm{Sym}^2X=\pp_\ba^2\ ;\ (q_1,q_2)\mapsto (z-q_1)(z-q_2),$$
namely the
conic $C:\{a_1^2-4a_0a_2=0\}$, which is the locus of double roots $q_1=q_2$: 
it is naturally parametrized by our initial base curve
$$X\to C\ ;\ q\mapsto (q^2:-2q:1).$$
Those lines $a_0b_0+a_1b_1+a_2b_2=0$ tangent to the conic are defined by the dual conic $C^*:\{b_1^2-b_0b_2=0\}$ 
(denoted by $\Pi$ in the previous section) which is also naturally parametrized by our initial base curve
$$X\to C^*\ ;\ z\mapsto (1:z:z^2).$$
The locus $q=t_i$ of poles provide $5$ special points on the conic $C$, namely $(a_0:a_1:a_2)=(t_i^2:-2t_i:1)$,
and we will denote by $\Delta_{i}:\{t_i^2a_2+t_ia_1+a_0=0\}$ the line tangent to $C$ at this point. Any two of those lines intersect
at a point $\Delta_i\cap \Delta_j=\{P_{i,j}\}$ (outside of $C$); we get $10$ special points with coordinates $(t_it_j:-t_i-t_j:1)$.
Passing to the dual picture, we get $5$
points $D_i:=\Delta_i^*$ on the dual conic $C^*$ defined by $(b_0:b_1:b_2)=(1:t_i:t_i^2)$ and 
$10$ lines, $\Pi_{i,j}:=P_{i,j}^*$ passing through both $D_i$ and $D_j$ with equation $t_it_jb_0-(t_i+t_j)b_1+b_2=0$
(see Figure \ref{figure:Duality5}).

\begin{figure}[htbp]
\begin{center}
\input{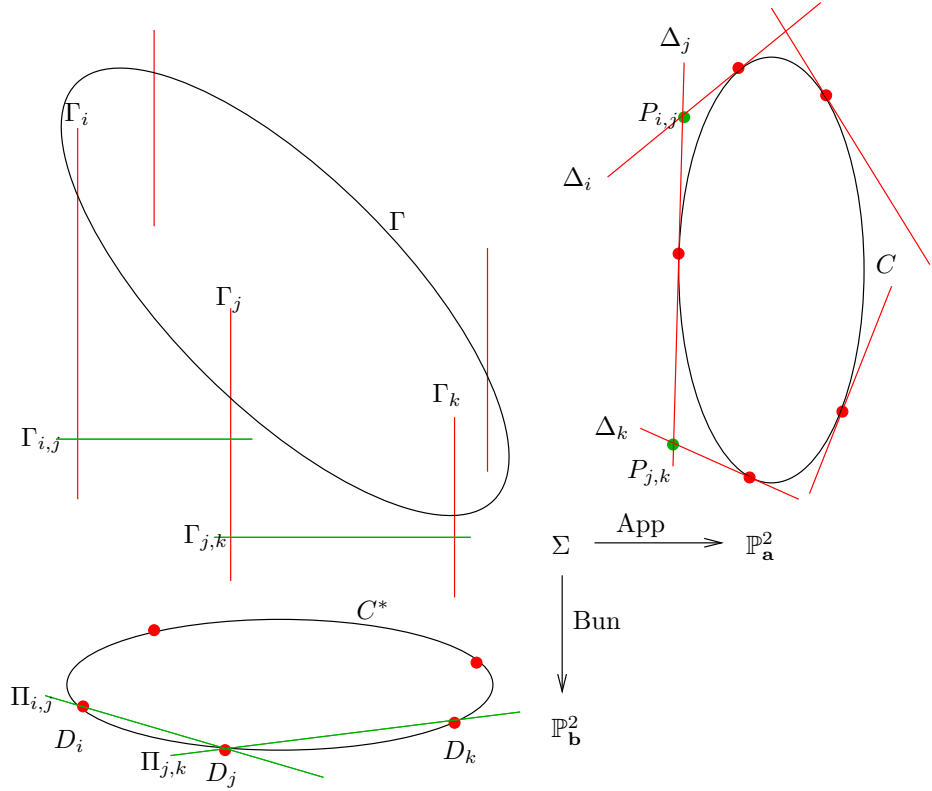}
\caption{Duality picture in the case $n=5$}
\label{figure:Duality5}
\end{center}
\end{figure}

Denote by $\Sigma\subset\mathbb P^2_{\mathbf{a}}\times\mathbb P^2_{\mathbf{b}}$ the incidence variety
defined by $a_0b_0+a_1b_1+a_2b_2=0$; recall that $\Sigma=M^{\bw}(\bt, \bnu)_H^0$ (see equation (\ref{eq:bd-bun})). 
The conic $C\subset \mathbb P^2_{\mathbf{a}}$ lifts-up as a rational curve
$\Gamma\subset\Sigma$ parametrized by 
$$\mathbb P^1\to \Gamma\ ;\ q\mapsto \left(\ (1:-2q:q^2)\ ,\ (q^2:q:1)\ \right)$$
(which projects down to both $C$ and $C^*$). It is defined by equations
$$a_1^2=4a_0a_2,\ \ \ a_0b_0=a_2b_2\ \ \ \text{and}\ \ \ 2a_2b_2+a_1b_1=0.$$
Inside $\Sigma$, we also get $5$ lines
$$\Gamma_i:=\Delta_i\times\{D_i\}$$
and $10$ more lines
$$\Gamma_{i,j}:=\{P_{i,j}\}\times\Pi_{i,j}.$$
All these $16$ curves intersect like the corresponding $16$
special rational curves in the Del Pezzo surface discussed in Section \ref{sec:Parabolic5} 
(the blow-up of $\mathbb P^2_{\mathbf{b}}$ at the $5$ points $D_i$, see picture \ref{figure:Duality5}); they moreover intersect transversally.

As we shall see, the locus of those connections that we have forgotten so far is given 
by points infinitesimally closed to some points of $\Sigma$, namely
\begin{itemize}
\item to $\Gamma_i$ for those connections on a bundle having
the parabolic $l_i\in\mathcal O_{X}$,
\item  to $\Gamma$ for those connections on $\mathcal O_{X}(1)\oplus\mathcal O_{X}(-2)$.
\end{itemize}
To recover the full moduli space, we will have (at least) to blow-up these curves.

\subsection{Those connections on $\mathcal O_{X}\oplus\mathcal O_{X}(-1)$ having a parabolic $l_i\in\mathcal O_{X}$}
To simplify formulae, set
$$\kappa_i:=\nu_i^+-\nu_i^-\ \ \ \text{for}\ i=0,1,t_1,t_2,\infty.$$
In order to recover such connections in our moduli space, we would like to construct, for each such connection
$(E,\boldsymbol{l}^0,\nabla^0)$, a deformation $(E,\boldsymbol{l}^t,\nabla^t)$ on the fixed bundle $E=\mathcal O_{X}\oplus\mathcal O_{X}(-1)$
such that it belongs to our main chart $M(\bt,\bnu)^0$ for $t\not=0$ (no parabolic $l_i^t$ is contained in $\mathcal O_{X}$).
We will do this for a connection $\nabla^0$ having $l_{t_1}^0\in\mathcal O_{X}$ and other $p_i$'s being generic 
(mainly $l_\infty\not\in\mathcal O_{X}$). After applying the elementary transformation $\mathrm{Elm}_\infty^+$,
we get a connection with parabolic structure in the chart $U$ with $u_1=\infty$. A deformation like above can be given by setting
$$c_1^t=-t\kappa_{t_1}+t^2\cdot c_1,\ \ \ c_2^t=c_2,\ \ \ u_1^t=\frac{1}{t} \ \ \ \text{and}\ \ \ u_2^t=u_2$$
(note that $\nabla_0-\frac{\kappa_{t_1}}{u_1}\Theta_1$ and $\frac{1}{u_1^2}\Theta_i$ have limit when $u_1\to\infty$).
By the way, we will get all connections for a generic parabolic structure $\boldsymbol{l}$ having $l_{t_1}\in\mathcal O_{X}$.
Going back with $\mathrm{Elm}_\infty^-$, we get a curve in $\mathbb P^2_{\mathbf{a}}\times\mathbb P^2_{\mathbf{b}}$
that tends to $\Sigma$ when $t\to 0$. The limit point is given by 
$$(a_2:a_1:a_0)\sim(1:-t_1-q:t_1q)\ \ \ \text{and}\ \ \ (b_2:b_1:b_0)\sim(t_1^2:t_1:0)$$
(i.e. we tend to a point of the special line $\Gamma_{t_1}$) with apparent points given by
$$q_1=t_1\ \ \ \text{and}\ \ \ q_2=\frac{t_2(c_2(u_2-1)-\rho-\kappa_{t_1}}{c_2(u_2-t_2)-\rho-\kappa_{t_1}}.$$
In order to distinguish between all connections having the same limit point (so far, $c_1$ does not appear for instance)
we have to blow-up $\Gamma_1$ and compute the limit point on the exceptional divisor $F_{t_1}$.
This latter one is parametrized by $q_2$ and the restriction of the projective coordinates
$$(u:v:w)\sim(\ b_2(t_1^2a_2+t_1a_1+a_0)\ :\ a_2(b_2-t_1b_1)\ :\ a_2(b_2-t_1^2b_0)\ ).$$
One easily check that, when $t\to 0$, the three entries above tend to $0$ but the triple projectively tends to
$$(u:v:w)\to(\ \frac{\kappa_{t_1}t_1^2t_2(t_2-1)}{c_2(u_2-t_2)-\rho-\kappa_{t_1}}\ :\ t_2(u_2-1)\ :\     (t_1+t_2)u_2-(t_1+1)t_2\ )$$
From the discussion of Section \ref{section:WallCrossing} and more particularly Section \ref{sec:Parabolic5},
it is not surprising that $u_2$, and thus the parabolic structure, is determined by the ratio $\frac{v}{w}=\frac{b_2-t_1b_1}{b_2-t_1^2b_0}$ which is also the coordinate 
for the blow-up of the point $D_{t_1}\in\mathbb P^2_{\mathbf{b}}$.
For $u_2$ fixed, we see that the parameter $c_2$ is determined by $q_2$, i.e. by the apparent map.
We still not see the parameter $c_1$ and cannot determine yet the limit connection. 
We have to blow-up once again. 

Precisely, we have now to blow-up the surface defined in $F_{t_1}$ by
$$(\rho+\kappa_{t_1})u+\kappa_{t_1}t_1(t_1+q_2)v-\kappa_{t_1}t_1q_2w=0.$$
One can check that the locus of those connections $p_{t_i}\in\mathcal O_{X}$
is parametrized by an open subset of the latter exceptional divisor $F_{t_1}'$.

\subsection{Those connections on $\mathcal O_{X}(1)\oplus\mathcal O_{X}(-2)$}

After an elementary transformation at the $5$ parabolics of the form
$$\mathrm{Elm}_0^+\circ\mathrm{Elm}_1^+\circ\mathrm{Elm}_\infty^+\circ\mathrm{Elm}_{t_1}^-\circ\mathrm{Elm}_{t_1}^-,$$ such a connection $(E=\mathcal O_{X}(1)\oplus\mathcal O_{X}(-2),\boldsymbol{l},\nabla)$ can be transformed into a trace-free connection $(E'=\mathcal O_{X}\oplus\mathcal O_{X},\boldsymbol{l}',\nabla')$ on the trivial bundle with the property that all parabolics $l_i'$ now lie along the diagonal section $\mathcal O_{X}(-1)\hookrightarrow\mathcal O_{X}\oplus\mathcal O_{X}$. We can now work in the chart $U$ of Section \ref{subsection:HiggsFields} and parametrize a small deformation, 
say $\nabla_t'$, on the trivial bundle whose parabolics
become generic (not lying anymore on a same $\mathcal O_{X}(-1)$) for $t\not=0$ and $\nabla_0'=\nabla'$. 
After coming back with the same $5$ elementary transformations (but opposite signs), we get a deformation $\nabla_t$ of connections
on the main bundle $E_t=\mathcal O_{X}\oplus\mathcal O_{X}(-1)$ for $t\not=0$ that tends to the initial connection
$\nabla_0=\nabla$ on the special bundle $E_0=\mathcal O_{X}(1)\oplus\mathcal O_{X}(-2)$. We thus get a curve
in our moduli space $\Sigma\subset\mathbb P^2_{\mathbf{a}}\times\mathbb P^2_{\mathbf{b}}$
that tends to $\Sigma$ when $t\to 0$. After Maple computations, we get the following.

First of all, the corresponding point in $\mathbb P^2_{\mathbf{a}}\times\mathbb P^2_{\mathbf{b}}$
tends to the ``conic'' $\Gamma$, the limit point depending on the first variation of the parabolics $l_i'$
at $t=0$: if we normalize so that the parabolic structure $\boldsymbol{l}'$ of $\nabla'$ is
$$(1:0),\ \ \ (1:1),\ \ \ (1:u_1'),\ \ \ (1:u_2')\ \ \ \text{and}\ \ \ (0:1)$$
(like notations of Section \ref{subsection:HiggsFields}) then the limit point on $\Gamma$ depends
on the slope $\lambda=\frac{u_2'-t_2}{u_1'-t_1}$ when $(u_1',u_2')\to(t_1,t_2)$.
Precisely, the limit point is 
$$\left(\ (1:-2q:q^2)\ ,\ (q^2:q:1)\ \right)\in\Gamma\ \ \ \text{where}\ \ \ 
q=\frac{t_1t_2((t_1-1)\lambda-(t_2-1))}{t_1(t_1-1)\lambda-t_2(t_2-1)}.$$ 
We fix this point 
from now on with genericity condition $q\not=0,1,t_1,t_2,\infty$.

At the neighborhood of $q$, the curve $\Gamma$ is given as complete intersection of
$$a_1^2=4a_0a_2,\ \ \ a_0b_0=a_2b_2\ \ \ \text{and}\ \ \ 2a_2b_2+a_1b_1=0.$$
Denote by $F$ the exceptional divisor obtained after blowing-up the curve $\Gamma$.
One can reduce our discussion to the hyperplane $2qa_2+a_1=0$ which is transversal to $\Gamma$.
Affine coordinates on $F$ are given by restricting the two rational functions
$$U=\frac{a_2}{b_2}\ \frac{b_2-qb_1}{q^2a_2-a_0}\ \ \ \text{and}\ \ \ V=\frac{a_2}{b_2}\ \frac{b_2-q^2b_0}{q^2a_2-a_0}.$$
Here, the strict transform of $\Sigma$ is given by $q^2(V-2U)+1=0$.
The limit of these two rational functions along a deformation $(E_t,\nabla_t)$ like above is
$$U\to \frac{1}{2q}\left(-\frac{2\rho+\kappa_0+5}{q}+\frac{\kappa_1-1}{q-1}+\frac{\kappa_{t_1}-1}{q-t_1}+\frac{\kappa_{t_2}-1}{q-t_2}\right)$$
and
$$V\to \frac{1}{q}\left(-\frac{\rho+\kappa_0+2}{q}+\frac{\kappa_1-1}{q-1}+\frac{\kappa_{t_1}-1}{q-t_1}+\frac{\kappa_{t_2}-1}{q-t_2}\right)$$
In particular, we can check that $q^2(V-2U)+1\to\rho+4\not=0$ for generic parameters $\kappa_i$.
This defines a curve $\Gamma'\subset F$ parametrized by $(1:-2q)=(a_2:a_1)$ on $F$ (or a single point since $q$ is fixed)
that we have to blow-up once again; let us call $F'$ the new exceptional divisor and still denote by $F$ the strict transform by abuse of notation. We then check by a direct computation that $F'\setminus (F\cap F')$ is the locus of those parabolic connections the bundle $E=\mathcal O_{X}(1)\oplus\mathcal O_{X}(-2)$. 

If we switch to Darboux coordinates, we can check that, along the above limit process, we get
$$q_1,q_2\to q\ \ \ \text{and}\ \ \ p_1,p_2\to\infty$$
with the constraints
$$p_1+p_2\to \frac{\rho}{\rho+4}\left(\frac{\kappa_0-1}{q}+\frac{\kappa_1-1}{q-1}+\frac{\kappa_{t_1}-1}{q-t_1}+\frac{\kappa_{t_2}-1}{q-t_2}\right)$$
and
$$\frac{p_1q_1+p_2q_2}{q}\to \frac{\rho}{\rho+4}\left(\frac{\kappa_0-3}{q}+\frac{\kappa_1-1}{q-1}+\frac{\kappa_{t_1}-1}{q-t_1}+\frac{\kappa_{t_2}-1}{q-t_2}\right).$$



\bibliographystyle{abbrv}

\bibliography{Garnier-2013-8-16.bib}

\end{document}